\documentclass{amsart}
\usepackage[utf8x]{inputenc}
\usepackage{etex}
\usepackage{adjustbox}
\usepackage{pstricks,  amssymb}
\usepackage{pict2e}
\usepackage{pst-plot}
\usepackage{xypic}
\usepackage{quiver}
\usepackage{tikz, tikz-cd}
\tikzcdset{scale cd/.style={every label/.append style={s}
      cells={nodes={scale=#1}}}}
\usetikzlibrary{arrows}
\usepackage{hyperref}
\usepackage{geometry}
\usepackage{rotating}
\usepackage{pst-node}
\usepackage{amsmath}
\usepackage[capitalize,nameinlink,noabbrev,nosort]{cleveref}

\newtheorem{thm}{Theorem}[section]
\newtheorem{theorem}[thm]{Theorem}
\newtheorem{mainthm}{Theorem}
\renewcommand{\themainthm}{\Alph{mainthm}}
\newtheorem{cor}[thm]{Corollary}
\newtheorem{corollary}[thm]{Corollary}

\newtheorem{lem}[thm]{Lemma}
\newtheorem{lemma}[thm]{Lemma}

\newtheorem{prop}[thm]{Proposition}
\newtheorem{proposition}[thm]{Proposition}
\theoremstyle{definition}
\newtheorem{defn}[thm]{Definition}

\newtheorem{example}[thm]{Example}
\newtheorem{definition}[thm]{Definition}
\newtheorem{rem}[thm]{Remark}  
\newtheorem{remark}[thm]{Remark}

\numberwithin{equation}{section}

\newcommand{\Q}{\Gamma}
\newcommand{\C}{\mathbb{C}}
\newcommand{\CC}{\mathbb{C}}
\newcommand{\OO}{\mathcal{O}}
\newcommand{\Z}{\mathbb{Z}}
\newcommand{\R}{\mathbb{R}}
\newcommand{\PP}{\mathbf{P}}
\newcommand{\V}{\mathcal{V}}
\newcommand{\nV}{\mathcal{V}_{\bullet}}
\newcommand{\Qbar}{\overline{Q}}
\newcommand{\vv}{\upsilon}
\newcommand{\uu}{\tilde{u}}

\newcommand{\Div}{\operatorname{Div}}
\newcommand{\CDiv}{\operatorname{CDiv}}
\newcommand{\Pic}{\operatorname{Pic}}

\newcommand{\conv}{\operatorname{Conv}}
\newcommand{\wt}{\operatorname{wt}}
\newcommand{\Arr}{\operatorname{Arr}}
\newcommand{\rect}{\operatorname{rect}}

\newcommand{\aff}{\operatorname{aff}}
\newcommand{\inn}{\operatorname{in}}
\newcommand{\out}{\operatorname{out}}
\newcommand{\F}{\mathcal{F}}
\newcommand{\FF}{\mathcal{FF}}
\newcommand{\Fl}{\operatorname{Fl}}
\newcommand{\NF}{\mathcal{N}}
\newcommand{\Newt}{\operatorname{Newt}}
\newcommand{\Trop}{\operatorname{Trop}}
\newcommand{\NP}{\operatorname{Root}}
\newcommand{\NNP}{\widetilde{\operatorname{Root}}}
\newcommand{\face}{\operatorname{face}}
\newcommand{\argmin}{\operatorname{argmin}}
\newcommand{\rank}{\operatorname{rank}}
\newcommand{\cone}{\operatorname{Cone}}
\newcommand{\trop}{\operatorname{trop}}

\def\x(#1){sin(#1)^3}
\def\y(#1){(13*cos(#1)-5*cos(2*#1)-2*cos(3*#1)-cos(4*#1))/16}
\psset{algebraic,plotpoints=100}

\begin{document}

\title[Root, flow, and order polytopes]{Root, flow, and order polytopes, with connections to toric geometry}

\author{K. Rietsch}
\address{Department of Mathematics,
            King's College London,
            Strand, London
            WC2R 2LS
            UK
}
\email{konstanze.rietsch@kcl.ac.uk}%
\author{L. Williams}%
\address{Department of Mathematics,
            Harvard University,
            Cambridge, MA
	    USA
}
\email{williams@math.harvard.edu}
%




\begin{abstract}
In this paper we study
the class of polytopes which can be obtained by 
taking the convex hull of some 
subset of the points $\{e_i-e_j \ \vert \ i \neq j\} 
\cup \{\pm e_i\}$ in $\R^n$, where $e_1,\dots,e_n$ is the standard
basis of $\R^n$.  
Such a polytope can be encoded by a 
quiver $Q$ with vertices $V \subseteq \{\vv_1,\dots,\vv_n\} \cup \{\star\}$,
where each edge $\vv_j\to \vv_i$ or $\star \to \vv_i$ or $\vv_i\to \star$ gives
rise to the point $e_i-e_j$ or $e_i$ or $-e_i$, respectively; 
we denote the corresponding polytope as $\NP(Q)$.
These polytopes have been studied extensively
under names such as \emph{edge polytope} and \emph{root polytope}.
We show that if the quiver $Q$ is strongly-connected
then the root polytope $\NP(Q)$ 
is \emph{reflexive} and \emph{terminal};
we moreover give
a combinatorial description of the facets of $\NP(Q)$.
We also show that if $Q$ is planar,
then  $\NP(Q)$ is (integrally 
equivalent to) the polar dual of the 
\emph{flow polytope} of the planar dual quiver $Q^{\vee}$.  
Finally we 
consider the case that $Q$ comes from the Hasse diagram of  a finite ranked poset $P$,
and show in this case that $\NP(Q)$ is polar dual to (a translation of) a 
\emph{marked
order polytope}. 
We then go on to study the toric variety $Y(\F_Q)$ associated to the face fan $\F_Q$ of $\NP(Q)$. If $Q$ comes from a ranked poset $P$  we give a combinatorial description of the Picard group of $Y(\F_Q)$, in terms of a new \emph{canonical ranked extension} of $P$, and we show that $Y(\F_Q)$ is a small partial desingularisation of the Hibi projective toric variety $Y_{\OO(P)}$ of the \emph{order polytope} $\OO(P)$. We show that $Y(\F_Q)$ has a small crepant toric resolution of singularities $Y(\widehat\F_Q)$, and as a consequence that the Hibi toric variety $Y_{\OO(P)}$ has a small resolution of singularities for any ranked poset $P$. 
These results have applications to mirror symmetry \cite{RW2}.
\end{abstract}


\maketitle

\setcounter{tocdepth}{1}
\tableofcontents

\section{Introduction}
We define a \emph{root polytope} to be the 
convex hull of some subset of 
the points $\{e_i-e_j \ \vert \ i \neq j\} 
\cup \{\pm e_i\}$ in $\R^n$, 
where $e_1,\dots,e_n$ is the standard
basis of $\R^n$.  
Root polytopes and their variants 
have been studied extensively in the literature, under different names.
In 2002, Ohsugi and Hibi \cite{Ohsugi} introduced the \emph{edge polytope} 
$\mathcal{P}_G$ of a 
directed graph $G$ on vertices $\vv_1,\dots,\vv_n$, defined as the convex hull
of the vertices $\{e_i-e_j\ \vert \ \vv_i \to \vv_j \text{ an edge  in }G\}$; 
they studied which orientations of the complete graph give rise to a 
Gorenstein Ehrhart ring $A(\mathcal{P}_G)$. 
There was subsequent work 
studying the roots of Ehrhart polynomials of $\mathcal{P}_G$ \cite{Matsui},
and 
 when
a directed graph yields a smooth Fano polytope $\mathcal{P}_G$ \cite{Higashitani}.
In 2009, Postnikov \cite{Postnikov2} studied 
the edge polytope
of a directed graph
$G$ on vertices $\vv_1,\dots,\vv_n$ in which
 edges can only connect $\vv_i \to \vv_j$ if $i<j$;
he called it the 
\emph{root polytope} because its vertices $e_i-e_j$ can be identified with positive roots in the type $A_{n-1}$ root system. He studied the volume
and triangulations of 
the root polytopes associated to bipartite graphs.
There has been much subsequent work on these polytopes, connecting 
them to subdivision algebras \cite{Meszarosroot},
subword complexes \cite{Escobar}, and R-systems \cite{GP}, 
studying them
in other types \cite{Ardilaroot},
and computing their faces \cite{Setiabrata} 
and 
$h^*$-vectors \cite{Kalman}.

In this paper we adopt the term ``root polytope'' for the slightly broader
class of polytopes whose vertices can be any subset of the points
$\{e_i-e_j \ \vert \ i \neq j\} 
\cup \{\pm e_i\}$ in $\R^n$.  
Our motivation for this work is mirror symmetry, more specifically, the study
of certain reflexive polytopes which come from the \emph{starred quiver}
encoding a Laurent polynomial superpotential for 
Schubert varieties \cite{RW2}.  
However, the beautiful 
properties of these reflexive polytopes hold in a very general setting, 
so we decided to present this material in a self-contained paper, which can 
be read independently of \cite{RW2}.  Some of our results are closely related
to results that have appeared before; we have done our best to include relevant citations
where appropriate.
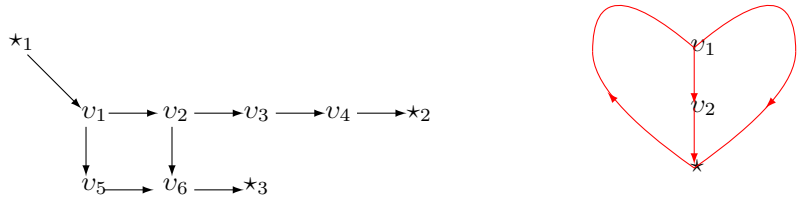
\begin{figure}[h]
\setlength{\unitlength}{1.2mm}
	\begin{picture}(50,14)
         \put(0,14){$\star_1$}
	 \put(8,6){$\vv_1$}
         \put(17,6){$\vv_2$}
          \put(26,6){$\vv_3$}
         \put(35,6){$\vv_4$}
         \put(44,6){$\star_2$}
         \put(8,-2){$\vv_5$}
         \put(17,-2){$\vv_6$}
	 \put(26,-2){$\star_3$}
         \put(2,13){{\vector(1,-1){6}}}
         \put(11,6.5){{\vector(1,0){5.5}}}
         \put(20.5, 6.5){{\vector(1,0){5.5}}}
         \put(10.5,-2){{\vector(1,0){5.5}}}
         \put(8.5,5){{\vector(0,-1){5.5}}}
         \put(18,5){{\vector(0,-1){5.5}}}
         \put(20.5,-2){{\vector(1,0){5.5}}}
         \put(29.5, 6.5){{\vector(1,0){5.5}}}
         \put(38.5, 6.5){{\vector(1,0){5.5}}}
 \end{picture}
	\hspace{1.5cm} \setlength{\unitlength}{1mm}
	\begin{picture}(25,25)(-12,0)
	 \put(0,16){$\vv_1$}
         \put(0,8){$\vv_2$}
	 \put(0,0){$\star$}
	\color{red}
		\qbezier(0.5,16.5)(14,28)(14,16)
	 \qbezier(14,16)(14,10)(0.5,0.5)
	 \qbezier(0.5,16.5)(-13,28)(-13,16)
	 \qbezier(-13,16)(-13,10)(0.5,0.5)
	 \put(-10,9.5){{\vector(-2,2.2){1}}}
	 \put(11,9.5){{\vector(-2,-2){1}}}
	 \put(0.5,16.5){{\vector(0,-1){7.5}}}
	 \put(0.5,8.5){{\vector(0,-1){7.5}}}
		\color{black}
 \end{picture}
 \vspace{15pt} 
\caption{Two starred quivers $Q$ and $Q'$.  
	The root polytope $\NP(Q)$ equals
	$\conv \{e_1, -e_4, -e_6, e_2-e_1, e_5-e_1, e_3-e_2, e_6-e_2, e_4-e_3, e_6-e_5\},$ which 
has $f$-vector $(9,34, 70, 84, 57, 18)$. 
	Meanwhile $\NP(Q')$ equals
	$\conv \{e_2-e_1, -e_2, e_1, -e_1\}$, and is a quadrilateral\label{fig:examplestarredquiver}}
  \end{figure}

\begin{definition}\label{def:starredquiver}
Let $Q$ be a quiver with vertices $\nV \sqcup \V_{\star}$ 
(where $\nV=\{\vv_1,\dots,\vv_n\}$ for $n \geq 1$ and 
$\V_{\star}=\{\star_1,\dots,\star_{\ell}\}$ for $\ell \geq 0$ are called the
	\emph{(normal) vertices} and 
\emph{starred vertices}, respectively),
and arrows $\Arr(Q) \subseteq (\nV \times \nV) \sqcup (\nV \times \V_{\star}) \sqcup (\V_{\star} \times \nV)$.  We will always assume that the underlying graph of any quiver is connected and has no loops. 
If $\ell >0$, we call $Q$ a \emph{starred quiver}.
\end{definition}

If any construction of a quiver results in duplicate arrows we remove these. If it results in an arrow
between two starred vertices, then we identify these vertices into one vertex and remove the arrow

\begin{definition}\label{def:NP}
Let $Q$ be a quiver or starred quiver as in \cref{def:starredquiver}
with arrows $\Arr(Q)$ 
and vertices
$\{\vv_1,\dots,\vv_n\} \cup \{\star_1,\dots,\star_{\ell}\}$.
We associate a point $u_a\in \R^n$ to each arrow $a$ as follows:
\begin{itemize}
	\item if $a: \vv_i \to \vv_j$, $u_a:=e_j-e_i$;
	\item if $a: \star_i \to \vv_j$, $u_a:= e_j$; and 
	\item if $a: \vv_i \to \star_j$, $u_a:= -e_i$.
\end{itemize}
We then define the \emph{root polytope} to be 
$$\NP(Q) = \conv\{u_a \ \vert \ a\in \Arr(Q)\} \subset \R^n,$$
 the convex hull of all the points $u_a$.
\end{definition}
See \cref{fig:examplestarredquiver} for examples.

\begin{remark}\label{rem:one}
If $Q$ has more than one starred vertex, 
we will often identify all starred vertices,
obtaining a related quiver $\Qbar$ with a unique starred vertex.
Clearly $\NP(Q) = \NP(\Qbar)$.  
\end{remark}

\begin{definition}\label{def:starconnected}
We say that a quiver $Q$ is \emph{strongly-connected}
if there is an oriented path from any vertex to any other vertex.
And we 
say that a starred quiver $Q$ is \emph{strongly-connected} 
	if, after identifying all starred vertices, the resulting quiver $\overline{Q}$
	is strongly-connected. 
\end{definition}

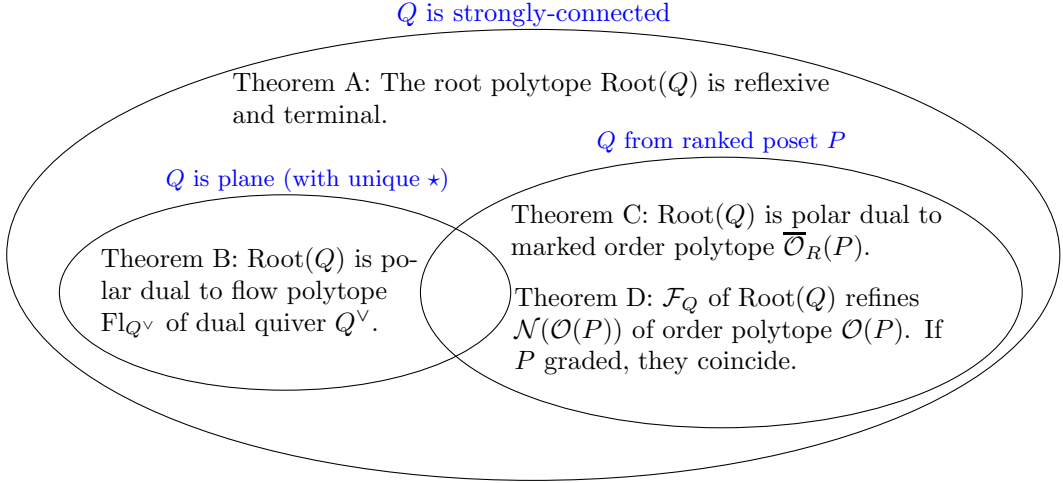
\begin{figure}
\begin{tikzpicture}
\draw (0,0) ellipse (7cm and 3cm);
    \draw (-3.3,-.5) ellipse(3cm and 1.3cm);
    \draw (2.5,-.5) ellipse(4cm and 1.8cm);
	\node 
	at (0,3.2) {\blue{$Q$ is strongly-connected}};
	\node at (-3,1) {\blue{\small{$Q$ is plane (with unique $\star$)}}};
	\node at (2.5,1.5) {\blue{\small{$Q$ from ranked poset $P$}}};
	\node[text width=8cm] at (0,2.1) {Theorem A: The root  
	polytope
	$\NP(Q)$ is reflexive and terminal.};
	\node[text width=4.5cm] at (-3.5,-0.5) {Theorem B: $\NP(Q)$ is 
	polar dual to flow 
	polytope $\Fl_{Q^\vee}$ \ \ of dual quiver 
	$Q^\vee$.};
	\node[text width=6cm] at (2.7,0.3) {Theorem C: $\NP(Q)$ is 
	polar dual to marked order polytope $\overline{\mathcal{O}}_R(P)$.};
	\node[text width=6.3cm] at (2.9,-1) {Theorem D: $\F_Q$ of 
	$\NP(Q)$ refines $\NF(\mathcal{O}(P))$ of order polytope $\mathcal{O}(P)$.  If $P$ graded, they coincide.};
\end{tikzpicture}
\caption{Overview of our main results  concerning root, flow and order polytopes: they appear 
	as \cref{t:reflexive},
	\cref{thm:dual},
	\cref{prop:NPreflexive},
	\cref{thm:refine}}
\label{fig:overview}
\end{figure}

We now explain our main results.  See \cref{fig:overview} for an overview.

 \begin{mainthm}\label{thm:reflexive}
 Let $Q$ be a strongly-connected quiver or starred quiver.
 Then the root polytope $\NP(Q)$ is reflexive and terminal 
 (cf \cref{def:reflexive} and 
 \cref{def:terminal}).  That is, the polar dual of $\NP(Q)$ is a lattice polytope,
 and the only lattice points of $\NP(Q) \cap \Z^n$ are the origin and the vertices.
 \end{mainthm}
\cref{thm:reflexive} appears as \cref{t:reflexive} and  
\cref{lem:0interior}.
Note that \cref{t:reflexive} also includes an explicit description
of the facets of $\NP(Q)$ in terms of \emph{facet arrow-labelings} of the quiver.
 The statement that a strongly-connected quiver $Q$ gives rise to 
 a reflexive and terminal root polytope
 $\NP(Q)$ already appeared in 
\cite[Proposition 1.4]{Higashitani} without a proof;  
\cite{Higashitani}
asserts
that the proof 
is the same as in the case of symmetric directed graphs, citing
\cite[Proposition 4.2]{Matsui} and \cite[Lemma 1.2]{Ohsugi}.
We provide an independent proof of \cref{thm:reflexive} for completeness.

Our next main results relate special cases of root polytopes to 
\emph{flow polytopes}
 \cite{Hille}, see \cref{def:flow}, and to \emph{(marked) order polytopes} \cite{order, marked}. These are two classes of polytopes which
have been extensively studied in toric geometry and combinatorics.  

\begin{mainthm}\label{mainthm:dual}
Let $Q$ be a strongly-connected 
	starred quiver with vertices $\{\vv_1,\dots,\vv_n\} \cup \{\star\}$ which is planar 
	(and comes with 
	a given plane embedding),
and let $Q^{\vee}$ be the dual quiver (which is plane, connected, and 
acyclic),  cf.
\cref{def:dual}. 
Then the root polytope $\NP(Q)$ 
is integrally equivalent to the polar dual of the flow polytope
$\Fl_{Q^{\vee}}$.
\end{mainthm}
\cref{mainthm:dual} appears later as \cref{thm:dual},
though we have restated it using  \cref{lem:bijection}. 
Note that  
\cref{thm:reflexive}  and
\cref{mainthm:dual} 
imply that 
the flow polytope associated to any plane acyclic quiver
is reflexive; this was previously shown more generally,  without the planarity assumption, in  \cite{Hille}.

We now consider the case that $Q$ is a starred quiver that comes from a poset.
\begin{definition}\label{quiverfromposet}
A \emph{starred poset} 
is a finite poset $P$ with a decomposition $P = P_{\bullet} \sqcup P_{\star}$ such that $P_{\star}$ contains the 
minimal and maximal elements of $P$, and no two elements of $P_\star$ are related by a covering relation.
We associate a starred quiver $Q:=Q_{(P_{\bullet},P_{\star})}$ to $P$ by letting 
$Q$ be the Hasse diagram of $P$, with all edges of the Hasse diagram directed upwards, from smaller to larger, and the  starred vertices of $Q$ given by $P_\star$. 
\end{definition}

When we associate a starred quiver to a poset, we will 
always assume that poset is connected.  Note that 
the \emph{bounded extension} of an arbitrary poset, 
defined below, will  always 
have these properties.

\begin{definition}\label{rem:quiverfromposet}
Let $P$ be a finite poset.  The \emph{bounded extension}
$\hat{P}:=P \cup \{\hat{0}, \hat{1}\}$ of $P$ is the poset containing $P$ where new elements
$\hat{0}$ and $\hat{1}$ are adjoined such that $\hat 0$ is the unique minimal and $\hat 1$ the unique maximal element of $\hat P$.  The associated starred quiver $Q_{\hat{P}}$ is constructed as in \cref{quiverfromposet} where $\hat P=P \sqcup P_{\star}$ with $P_{\star} = \{\hat{0}, \hat{1}\}$.
\end{definition}

We note that the 
root polytopes associated to quivers $Q_{\hat{P}}$ were studied in \cite{HH} 
where it was proved that they are always terminal and reflexive -- 
a special case of \cref{thm:reflexive} above.

\begin{mainthm} \label{mainthm:poset}
	Let $Q:=Q_{(P_{\bullet}, P_{\star})}$ be a starred quiver that comes from a  starred poset $P=P_{\bullet} \sqcup P_{\star}$  as in \cref{quiverfromposet}.
Suppose that 
$P$ is ranked with rank function $R$ (cf \cref{def:rank}).
Then the root polytope $\NP(Q)$
of the starred quiver $Q$ 
is polar dual to the (translated) \emph{marked order polytope}
$\overline{\OO}_R(P)$
from \Cref{p:reflexive}.
It follows that 
$\overline{\OO}_R(P)$ is reflexive.
\end{mainthm}

\cref{mainthm:poset} appears as \cref{prop:NPreflexive}.
The fact that  
$\overline{\OO}_R(P)$ is reflexive was known earlier, see
  \cite[Proposition 3.1 and Theorem 3.4]{FFP}.
  The related statement that the 
  order polytope of a poset $P$ is Gorenstein if and only if $P$ is graded
  goes back to 
  \cite{hibi1987distributive, Stanley78}.

The following result 
relates the root polytope to the (ordinary) 
order polytope, and appears as
\cref{thm:refine}. 

\begin{mainthm}\label{thm:D}
Let $P$ be a finite ranked poset and let 
$Q:=Q_{\hat{P}}$ be the starred quiver associated to the bounded extension
	$\hat{P}$ of $P$ (as in \cref{rem:quiverfromposet}).
Then the face fan $\F_Q$ of the root polytope $\NP(Q)$
of the starred quiver $Q$
refines the (inner) normal fan $\NF(\OO(P))$ of 
the \emph{order polytope} $\mathcal{O}(P)$ of $P$:
the two fans have the same set of rays, and each 
maximal cone of $\NF(\OO(P))$ is a union of maximal cones of 
	$\F_{Q_P}$.  And if $P$ is a graded poset, then 
	$\F_{Q_P}$ coincides with 
 $\NF(\OO(P))$.
\end{mainthm}
\begin{remark}
Note that it is possible for $P$ to be a ranked poset but for
$\hat{P}$ to fail to be a ranked poset, see 
	   \cref{fig:starredquiver}.
	Thus the hypotheses of \cref{mainthm:poset} and \cref{thm:D} 
	are subtly different.
\end{remark}

In light of 
\cref{mainthm:dual} and 
\cref{mainthm:poset}, it is interesting to consider the situation
where $P$ is a poset with a planar embedding. 
This setting was studied by  Meszaros-Morales-Striker 
\cite[Theorem 3.14]{MMS}, who 
 showed that a flow polytope $\Fl(Q)$ 
 is (integrally equivalent to) the order polytope $\OO(P)$
 of a poset $P$ exactly 
 when $Q$ is a planar embedding of $P$; this result 
was extended to marked order polytopes  in \cite{LMS}.

We next turn our attention to the toric variety $Y(\F_Q)$ associated to the face fan 
$\F_Q$
of the root polytope.
When $Q$ is strongly connected, 
\cref{thm:reflexive} implies that 
$Y(\F_Q)$ is a Gorenstein Fano variety with at most terminal singularities. 
The fact that $Y(\F_Q)$ is very close to being smooth is further underlined by the following result, which appears as \cref{p:desing}. 

\begin{mainthm}\label{p:desing0} 
Let $Q$ be a strongly-connected starred quiver,
and let $\F_Q$ be the  face fan of the root polytope $\NP(Q)$. 
 There exists a refinement $\widehat \F_Q$ of $\F_Q$ such that the resulting morphism $Y(\widehat \F_Q)\to Y(\F_Q)$ is a small crepant toric desingularisation. 
\end{mainthm}

When our quiver $Q$ comes from a ranked poset $P$ as in \cref{thm:D},  there is a third fan to consider. Namely, \cref{thm:D} tells us that  the normal fan 
$\NF(\OO(P))$ 
of the order polytope 
is refined by $\F_Q$, while \cref{p:desing0} tells us that $\F_Q$ is refined by $\widehat\F_Q$. All three fans share the same set of rays. The toric variety associated to  $\NF(\OO(P))$ is the 
so-called Hibi projective toric variety $Y_{\OO(P)}$ associated to the poset $P$, and the two results join together to give a small desingularisation of $Y_{\OO(P)}$ that goes via $Y(\F_Q)$, 
\[
Y(\widehat\F_Q)\longrightarrow Y(\F_Q)\longrightarrow Y_{\OO(P)}.
\]
In other words, as a corollary of \cref{thm:D} and \cref{p:desing0}, any Hibi toric variety associated to a ranked poset has a small toric desingularisation. 

Our final main result is a combinatorial description of the Picard group of $Y(\F_Q)$ in the case where $Q$ comes from a ranked poset $P$. Recall that the poset $P$ has its bounded extension $\hat P$ with one unique maximal element. It also has a `maximal' extension $P_{\max}$ where each maximal element $m$ is covered by a separate adjoined element $\hat 1_m$. In between these two extensions we define a new \textit{canonical extension} $\bar P$ in \cref{d:canonical-extension}, and we prove the following theorem, which appears as 
\cref{t:canonicalPCartier} and \cref{p:Hibidesing}.

\begin{mainthm}\label{t:Pic} Let $P$ be a ranked poset, 
 let $Y(\F_Q)$ be the toric variety associated to 
the quiver $Q = Q_{\bar P}$,
and let $Y(\widehat\F_Q)$ be 
its desingularisation from \cref{p:desing0}. The Picard rank of $Y(\F_Q)$ is equal to the number of maximal elements in the canonical extension $\bar P$ of $P$, and the Picard rank of $Y(\widehat\F_Q)$ is equal to the number of maximal elements in $P$. 
\end{mainthm}  

Note that the Hibi toric variety $Y_{\OO(P)}$ has Picard rank equal to $1$, see \cite[Section~2.3]{Miura:CYinHibi}, coinciding with the number of maximal elements in $\hat P$. We find that the three toric varieties $Y_{\OO(P)},Y(\F_Q)$ and $Y(\widehat \F_Q)$ constructed out of $P$ relate naturally to the three extensions $\hat P,\bar P$ and $P_{\max}$ of the ranked poset~$P$.

At the heart of all of our results are the reflexive polytopes $\NP(Q)$ associated to strongly-connected starred quivers. Reflexive polytopes were introduced by Batyrev \cite{Batyrev} in the study of
mirror symmetry for toric varieties and,
as mentioned earlier, this work was motivated by our concurrent work on mirror symmetry for Schubert varieties
and their toric degenerations \cite{RW2}, see \cref{sec:toric}
for more details.
However, the results of this paper are combinatorial in nature.
While some of the results presented here have appeared before (sometimes in special
cases or without proofs), we hope that this exposition will be useful
for illuminating the connections between root polytopes,
flow polytopes, and (marked) order polytopes, and their roles within toric geometry. 

The structure of this paper is as follows. 
In \cref{sec:root} 
we prove that if $Q$ is strongly-connected,
$\NP(Q)$ is reflexive and terminal.  We also describe the facets
of $\NP(Q)$.
In \cref{sec:planar} we consider the case that $Q$ has a planar embedding,
and we show that in this case, $\NP(Q)$ is (integrally equivalent to the)
polar dual to the flow
polytope of the planar dual quiver $Q^{\vee}$.
In \cref{sec:markedorder} we consider
the case that $Q$ comes from a ranked poset, and we relate
$\NP(Q)$ to a marked order polytope.  We also relate the face fan $\F_Q$ of 
$Q$ to the (inner) normal fan of the corresponding order polytope.
Finally in \cref{sec:toric} we discuss the connection to mirror symmetry
and toric geometry:
we show that when $Q$ is a strongly-connected
starred quiver, 
there is a {small} toric desingularisation $Y(\widehat\F_Q)$ of 
the toric variety $Y(\F_Q)$,
and we compute the Picard group of $Y(\F_Q)$. We give a particularly explicit description in the case where $Q$ comes from a ranked poset.

\section{Root polytopes}\label{sec:root}

In this section 
we show that 
when $Q$ is a quiver or starred quiver that is  strongly connected,
the polytope $\NP(Q)$ is reflexive and terminal.
In this case we also describe the facets of $\NP(Q)$ in terms of 
certain labelings of the arrows of $Q$.

\subsection{Preliminaries}

\begin{definition}
We say that two integral polytopes 
	$\PP_1 \subset \R^n$ and $\PP_2 \subset \R^m$ are \emph{integrally
	equivalent} if there is an affine transformation 
	$\phi:\R^n \to \R^m$ whose restriction to $\PP_1$
	is a bijection $\phi:\PP_1 \to \PP_2$ that preserves the lattice,
	i.e. $\phi$ is a bijection between 
	$\Z^n \cap \aff(\PP_1)$ and $\Z^m \cap \aff(\PP_2)$, where
	$\aff(\cdot)$ denotes the affine span.  The map $\phi$
	is then an \emph{integral equivalence}.  
\end{definition}
We note that integrally equivalent polytopes (sometimes called \emph{isomorphic} or 
\emph{unimodularly equivalent}) have the same Ehrhart polynomials
and hence the same volume.

\begin{definition}
Suppose that $\PP\subset \R^n$ is a lattice polytope of full dimension $n$
which contains the origin in its interior. Then
the \emph{polar dual polytope} of $\PP$ is
	\begin{equation}\label{eq:polar}
\PP^*:= \{y \in (\R^n)^* \ \vert \ x \cdot y \geq -1
  \text{ for all }x\in \PP\}.
	\end{equation}
\end{definition}
\begin{remark}
There are two common definitions given for the polar dual.
Our convention in \eqref{eq:polar} is consistent with 
the conventions of Polymake.  
The definition of polar dual from \cite{Ziegler} is 
\begin{equation}\label{eq:polar2}
\PP^\Delta:= \{y \in (\R^n)^* \ \vert \ x \cdot y \leq 1
 \text{ for all }x\in \PP\}.
\end{equation}
The polytope $\PP^{\Delta}$ is simply the negative of $\PP^*$.
\end{remark}

\begin{definition} \label{def:reflexive}
A \emph{reflexive polytope of dimension $n$} is a lattice polytope of full 
dimension $n$ such that 
its polar dual is also a lattice polytope, i.e. it is 
bounded and has vertices with integer coordinates.
\end{definition}

More generally, we will use the word \emph{reflexive} for any polytope that is integrally equivalent
to a reflexive polytope.

\begin{definition}\label{def:terminal}
A lattice polytope $\PP \subset \Z^n$ is called \emph{terminal}
if $\mathbf{0}$ and the vertices are the only lattice points in 
	$\PP \cap \Z^n$ (with $\mathbf{0}$ in the interior).
\end{definition}

\subsection{Examples of root polytopes}

Recall the definition of root polytope from \cref{def:NP}.
We can associate natural starred quivers and thereby root polytopes to both acyclic quivers and to posets, in the following way.

\begin{example}[Strongly-connected quiver from an acyclic quiver]\label{ex:acyclic}
We can get a strongly-connected starred quiver from any acyclic quiver:
we simply designate each sink and source as a starred vertex, 
	  see \cref{fig:starredquiver}.
\end{example}

\begin{example}[Strongly-connected quiver from a  poset]\label{def:posetquiver}
As in \cref{quiverfromposet} and \cref{rem:quiverfromposet},
we can get a strongly-connected starred quiver from a starred poset 
(the underlying quiver will be connected as long as e.g.
the poset has a unique minimal or unique maximal element),
or from the bounded extension of any finite poset $P$.
The latter construction is illustrated at the right of \cref{fig:starredquiver}. 
\end{example}

\begin{figure}
	\[
		\adjustbox{scale=.7,center}{\begin{tikzcd}
               & \bullet && \bullet &&&& \bullet && \star \\
               && \bullet &&&&&& \bullet \\
               \bullet && \bullet && \bullet && \star && \bullet && \bullet
               \arrow[from=1-2, to=1-4]
               \arrow[from=1-8, to=1-10]
               \arrow[from=2-3, to=1-2]
               \arrow[from=2-3, to=1-4]
               \arrow[from=2-9, to=1-8]
               \arrow[from=2-9, to=1-10]
               \arrow[from=3-1, to=1-2]
               \arrow[from=3-1, to=2-3]
               \arrow[from=3-1, to=3-3]
               \arrow[from=3-3, to=1-2]
               \arrow[from=3-3, to=1-4]
               \arrow[from=3-3, to=2-3]
               \arrow[from=3-3, to=3-5]
               \arrow[from=3-5, to=1-4]
               \arrow[from=3-5, to=2-3]
               \arrow[from=3-7, to=1-8]
               \arrow[from=3-7, to=2-9]
               \arrow[from=3-7, to=3-9]
               \arrow[from=3-9, to=1-8]
               \arrow[from=3-9, to=1-10]
               \arrow[from=3-9, to=2-9]
               \arrow[from=3-9, to=3-11]
               \arrow[from=3-11, to=1-10]
               \arrow[from=3-11, to=2-9]
		\end{tikzcd}
	\hspace{2.5cm}
		\begin{tikzcd}
               &&& \star \\
               \bullet &&& \bullet \\
               \bullet & \bullet && \bullet & \bullet \\
               \bullet & \bullet && \bullet & \bullet \\
               \bullet &&& \bullet \\
               &&& \star
               \arrow[from=2-4, to=1-4]
               \arrow[from=3-1, to=2-1]
               \arrow[from=3-4, to=2-4]
               \arrow[from=3-5, to=1-4]
               \arrow[from=4-1, to=3-1]
               \arrow[from=4-1, to=3-2]
               \arrow[from=4-2, to=3-2]
               \arrow[from=4-4, to=3-4]
               \arrow[from=4-4, to=3-5]
               \arrow[from=4-5, to=3-5]
               \arrow[from=5-1, to=4-1]
               \arrow[from=5-1, to=4-2]
               \arrow[from=5-4, to=4-4]
               \arrow[from=5-4, to=4-5]
               \arrow[from=6-4, to=5-4]
		\end{tikzcd}
		}
	\]
	\caption{An acyclic quiver and the corresponding starred quiver; the 
	Hasse diagram of a ranked poset and the starred quiver associated to its bounded extension.
	Note that the bounded extension is not ranked\label{fig:starredquiver}}
\end{figure}
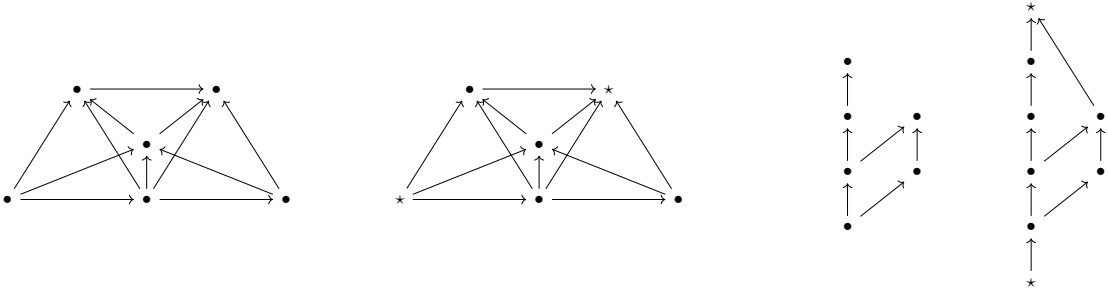

\begin{lemma}\label{lem:project}
Let $Q$ be a quiver with no starred vertices.  Choose one of its vertices, say $\vv_1$,
and let $Q_{\star}$ be the starred quiver obtained from $Q$ by replacing $\vv_1$ with a starred vertex $\star$.
Then $\NP(Q)$ is integrally equivalent to $\NP(Q_{\star})$.
\end{lemma}
\begin{proof}
If $Q$ has normal vertices $\{\vv_1,\dots,\vv_n\}$ but no starred vertices, then 
$\NP(Q)$ lies in the hyperplane $x_1+\dots+x_n=0$.
The map $\phi: \R^n \to \R^{n-1}$ which sends $(x_1,x_2,\dots,x_n) \mapsto (x_2,\dots,x_n)$
restricts to an integral equivalence mapping $\NP(Q)$ to $\NP(Q_{\star})$.
\end{proof}

Whenever the underlying graph of $Q$ has only normal vertices $\{\vv_1,\dots,\vv_m\}$ but no starred vertices,  the root polytope lies in the hyperplane $x_1  + \dots + x_m = 0$. If there is at least one starred vertex, we have the following lemma. 

\begin{lemma}\label{l:full-dimensional}
Let $Q$ be a starred quiver with $n$ normal vertices $\{\vv_1,\dotsc,\vv_n\}$.
Then $\NP(Q)$ is full-dimensional in $\R^n$. 
\end{lemma}
\begin{proof}
Recall that we assume that the underlying graph of any quiver is connected and has no loops. Since $Q$ is starred, it has at least one starred vertex. We may identify the starred vertices by \cref{rem:one}, and assume that $Q$ has a unique starred vertex, that we call $\star$. 
Now for any normal vertex $\vv_i$ we have a path in the underlying graph of $\overline Q$ that starts at the starred vertex $\star$ and ends at $\vv_i$. Let $(a_1,\dotsc, a_k)$ be the associated sequence of arrows of $Q$. Consider the element of $\R^n$ given by the signed sum $\vv_\pi:=\sum_{i=1}^k \varepsilon_i u_{a_i}$ of vertices of $\NP(Q)$, where  $\varepsilon_i=1$ if $a_i$ is oriented in the direction of the path $\pi$ and $\varepsilon_i=-1$ otherwise. In terms of the standard basis of $\R^n$ we have that $\vv_{\pi}=e_i$. Thus $e_i$ lies in the span of $\NP(Q)$. This holds for all $i$ and proves that $\NP(Q)$ is full-dimensional. 
\end{proof}

\begin{lemma}\label{lem:vertexarrow}
Let $Q$ be a starred quiver 
with arrows $\Arr(Q)$
and vertices
$\{\vv_1,\dots,\vv_n\} \cup \{\star\}$.
Then each $u_a$ for $a\in \Arr(Q)$ is a vertex of $\NP(Q)$, and 
the vertices of $\NP(Q)$ are in bijection with the 
arrows $\Arr(Q)$.
\end{lemma}

\begin{definition}\label{def:root}
Let 
$A^{(n)}:=
	\conv\{\{\pm e_1,\dots,\pm e_n\} \sqcup \{\pm (e_i - e_j) \ \vert \ 1 \leq i < j \leq n\}\} \subset 
\R^n$.  
\end{definition}
Note that  by 
\Cref{lem:project},
$A^{(n)}$ is isomorphic to the \emph{root polytope} of type $A$ in $\R^{n+1}$
(which is defined as the convex hull of all points $\pm (e_i-e_j)$, and hence
lies in the hyperplane $\sum x_i = 0$).

\begin{proof}
	[Proof of \cref{lem:vertexarrow}]
Note that each of the $(2n)+n(n-1)$ points in \cref{def:root}
is a vertex of $A^{(n)}$; 
we can easily see this by 
choosing for each  $u$ in $S$ an appropriate linear functional which is maximized at $u$.
In particular, $\pm e_1$ is a vertex because 
the linear functional $\lambda:(x_1,\dots,x_n) \mapsto \pm (10x_1 + x_2 + \dots + x_n)$ is 
maximized at $\pm e_1$,
while $\pm (e_1-e_2)$ is a vertex because the linear functional 
$\lambda:(x_1,\dots,x_n) \mapsto \pm (10x_1 -10 x_2)$ is 
	maximized at $\pm (e_1-e_2)$.

But now $\{u_a \ \vert \ a\in \Arr(Q)\}$ is a subset of the vertices of 
$A^{(n)}$, and hence 
each $u_a$ must be a vertex of $\NP(Q)$.
\end{proof}

\begin{remark}\label{r:singlestar} 
	Let $Q$ be any (possibly starred) quiver. If $Q$ has multiple starred vertices then 
we may identify all of the starred vertices to obtain a starred quiver with a single starred vertex as in \cref{rem:one}. If  $Q$ has no starred vertices then we may apply \cref{lem:project} to obtain a starred quiver with one starred vertex. In either case the root polytope of the new quiver is integrally equivalent to $\NP(Q)$. In light of this observation and \cref{l:full-dimensional}, we may at times prefer to work with quivers which have precisely one starred vertex. 
\end{remark}

\begin{proposition}\label{p:NPlatticepts}
Let $Q$ be a (possibly starred) quiver with normal vertices
	$\nV=\{\vv_1,\dots,\vv_n\}$.  Then the only  possible
	lattice points of $\NP(Q)$ are $\mathbf{0}$ and its vertices
	$u_a$ for $a\in \Arr(Q)$. 
\end{proposition}
\begin{proof}
	Choose a lattice point $p=(p_1,\dots,p_n)\in \NP(Q)$.
	We can write it as 
	\begin{equation}\label{p:1}
		p=\sum_{a\in \Arr(Q)} m_a u_a,
	\end{equation}
	where each $m_a \geq 0$, and $\sum_a m_a=1$.
If we then replace each $u_a$ by its expression in terms of $e_i$'s,
we get 
	\begin{equation}\label{p:2}
		p=(p_1,\dots,p_n)=\sum_{i=1}^n (\inn(i)-\out(i)) e_i,
	\end{equation}
	where $\inn(i) = \sum_a m_a$,
	where the sum is over all arrows $a$ pointing towards
	$\vv_i$, and $\out(i)=\sum_a m_a$, where the sum is 
	over all arrows $a$ pointing away from $\vv_i$.

Assume that $p$ is not one of the vertices $u_a$. 
Then we must have 
that  $m_a <1$ for all arrows $a$.
Since $\sum_a m_a=1$, we must have 
$0 \leq \inn(i) \leq 1$ and $0 \leq \out(i) \leq 1$ for all $i$.
	Moreover if $\inn(i)=1$ (respectively, $\out(i)=1$)
	then $\out(i)=0$ (respectively, $\inn(i)=0$).
Since each coordinate $p_i\in \Z$, it then follows that 
	$p_i\in \{0,1,-1\}$ for all $1 \leq i \leq n$.

Now fix some $i$, and suppose that $p_i=1$ (the proof in 
the case where $p_i=-1$
is analogous so we omit it).  Then 
$\inn(i)=1$ and $\out(i)=0$.  Since $\sum_a m_a=1$, this implies that 
for each arrow $a$ not pointing to $\vv_i$, we have that $m_a=0$.

Since $\inn(i)=1$, we know there exists some arrow $a$ pointing to $\vv_i$ such that $m_a >0$.
If the only such arrow(s) $a$ starts from a starred vertex $\star$, then 
we have that $p$ has the form $u_a$, where $a$ is the arrow from 
$\star$ to $\vv_i$, which contradicts our assumption that $p$ 
does not have the form $u_a$.
Therefore we must have an arrow $a:\vv_k \to \vv_i$ with $m_a>0$;
recall that we also know that $m_a<1$.  But then 
	$p_k = \inn(k)-\out(k)=0-m_a \notin \Z$.  This is a contradiction.
\end{proof}

The previous proposition immediately implies the following result.
\begin{corollary}\label{cor:terminal}
Let $Q$ be a (possibly starred) quiver with normal vertices
$\nV=\{\vv_1,\dots,\vv_n\}$.  If $\mathbf{0}$ lies in the 
relative interior of $\NP(Q)$, then $\NP(Q)$ is terminal. \qed 
\end{corollary}

\begin{proposition}\label{lem:0interior}
Let $Q$ be a (possibly starred) quiver and let $Q$ be strongly-connected.
Then the polytope $\NP(Q)$
is terminal:
	in particular, it contains $\mathbf{0}$ in its relative
	interior. 
	\end{proposition}
\begin{proof}
By \cref{r:singlestar}	we can assume that 
$Q$ 
has arrows $\Arr(Q)$
and vertices
$\V=\{\vv_1,\dots,\vv_n\} \cup \{\star\}$.
Choose any $\vv_h\in \V$, and 
let $p$ be a path from  $\star$ to $\vv_h$,
with arrows $a_1,\dots,a_k$.
Note that the sum of the corresponding 
points $u_{a_1} + \dots +u_{a_k}$ equals $e_h$,
and hence $\frac{1}{k} e_h$ lies in $\NP(Q)$.
Similarly, if we choose a path from $v_h$ to $\star$,
 with arrows $a'_1,\dots,a'_{\ell}$,
we have 
$u_{a'_1} + \dots + u_{a'_{\ell}} = -e_h$,
	and hence $-\frac{1}{\ell} e_h$ lies in $\NP(Q).$
Since $\frac{1}{k} e_h$ and
	$-\frac{1}{\ell} e_h$ lie in $\NP(Q)$, so does $\mathbf{0}$.

The above argument works for any $\vv_h\in \V$, so 
for each $h\in \{1,2,\dots,n\}$,
the interval $[-\mu e_h, \mu e_h]$ lies in $\NP(Q)$, 
for some small $\mu>0$.  This implies that $\mathbf{0}$ lies 
	in the interior of $\NP(Q)$.  And now by 
	\cref{cor:terminal}, $\NP(Q)$ is terminal.
\end{proof}

\subsection{Faces of root polytopes}

We next turn our attention to faces of $\NP(Q)$, where
$Q$ is a starred quiver.
We will start by identifying each point $u$ of $\NP(Q)$ with a 
\emph{$\bullet$-labeling}
$L: \nV \to \R$ of 
$Q$.  
(Abusing notation, we will also identify 
a $\bullet$-labeling 
$L': \nV \to \R$ with a linear functional, which we evaluate 
on points $u$ of $\NP(Q)$ by taking the dot product.)
Our main result is \cref{t:reflexive}, which characterizes
the facets of the root polytope of a strongly-connected 
starred quiver; as a consequence, we also prove that this polytope is 
reflexive. 

Note that in the case that $Q$ is an acyclic quiver, the faces of 
$\NP(Q)$ were characterized in \cite{Setiabrata}.  While we can always
associated a strongly-connected starred quiver to an acyclic quiver (cf. 
\cref{ex:acyclic}), a strongly-connected starred quiver may  have cycles.

To describe the faces and facets of $\NP(Q)\subset\R^n$ we will use two types of coordinates. 

\begin{definition}[Vertex coordinates and arrow coordinates]\label{def:coords}
Let $Q$ be a starred quiver as in \Cref{def:starredquiver}.
We define  a \emph{$\bullet$-labeling} to be a map 
$L: \nV\to \R$ from the set of normal vertices $\V_\bullet=\{\vv_1,\dots,\vv_n\}$ to $\R$;
if $L(\vv_i)=c_i$, we identify $L$ with 
the point $c_1 e_1 + \dots c_n e_n=(c_1,\dots,c_n)\in \R^n$.
And we define an \emph{arrow labeling} to be a map
$M: \Arr(Q) \to \R$.
We associate to each $\bullet$-labeling $L$ an \emph{arrow labeling}
$M_L: \Arr(Q)\to \R$ as follows:
\begin{enumerate}
\item if $a: \vv_i \to \vv_j$, $M(a)=L(\vv_j)-L(\vv_i)$;
\item if $a: \star_i \to \vv_j$, $M(a)=L(\vv_j)$; and 
\item if $a: \vv_i \to \star_j$, $M(a)=-L(\vv_i)$.
\end{enumerate}
We may think of each starred vertex as getting labeled by a $0$, i.e.
 we may extend $L$ to a labeling of $\V_{\star}$ by setting $L(\star_i)=0$,
and then Cases (2) and (3) above are special cases of (1).
\end{definition}

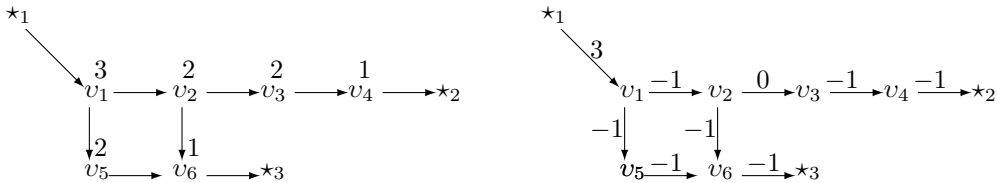
\begin{figure}[h]
\setlength{\unitlength}{1.3mm}
\begin{center}
 \begin{picture}(50,14)
         \put(0,14){$\star_1$}
	 \put(8,6){$\vv_1$}
	 \put(9,8){$3$}
         \put(17,6){$\vv_2$}
         \put(18,8){$2$}
          \put(26,6){$\vv_3$}
          \put(27,8){$2$}
         \put(35,6){$\vv_4$}
         \put(36,8){$1$}
         \put(44,6){$\star_2$}
         \put(8,-2){$\vv_5$}
         \put(9,0){$2$}
         \put(17,-2){$\vv_6$}
         \put(18.5,0){$1$}
	 \put(26,-2){$\star_3$}
         \put(2,13){{\vector(1,-1){6}}}
         \put(11,6.5){{\vector(1,0){5.5}}}
         \put(20.5, 6.5){{\vector(1,0){5.5}}}
         \put(10.5,-2){{\vector(1,0){5.5}}}
         \put(8.5,5){{\vector(0,-1){5.5}}}
         \put(18,5){{\vector(0,-1){5.5}}}
         \put(20.5,-2){{\vector(1,0){5.5}}}
         \put(29.5, 6.5){{\vector(1,0){5.5}}}
         \put(38.5, 6.5){{\vector(1,0){5.5}}}
 \end{picture}\hspace{.5cm}
 \begin{picture}(50,14)
         \put(0,14){$\star_1$}
	 \put(8,6){$\vv_1$}
	 \put(5,10){$3$}
	 \put(11,7){$-1$}
         \put(17,6){$\vv_2$}
         \put(22,7){$0$}
          \put(26,6){$\vv_3$}
          \put(29,7){$-1$}
         \put(35,6){$\vv_4$}
         \put(38,7){$-1$}
         \put(44,6){$\star_2$}
         \put(8,-2){$\vv_5$}
         \put(5,2){$-1$}
         \put(8,-2){$\vv_5$}
         \put(14.5,2){$-1$}
         \put(17,-2){$\vv_6$}
         \put(11,-1.5){$-1$}
         \put(21,-1.5){$-1$}
	 \put(26,-2){$\star_3$}
         \put(2,13){{\vector(1,-1){6}}}
         \put(11,6.5){{\vector(1,0){5.5}}}
         \put(20.5, 6.5){{\vector(1,0){5.5}}}
         \put(10.5,-2){{\vector(1,0){5.5}}}
         \put(8.5,5){{\vector(0,-1){5.5}}}
         \put(18,5){{\vector(0,-1){5.5}}}
         \put(20.5,-2){{\vector(1,0){5.5}}}
         \put(29.5, 6.5){{\vector(1,0){5.5}}}
         \put(38.5, 6.5){{\vector(1,0){5.5}}}
 \end{picture}
\end{center}
	\caption{A $\bullet$-labeling and the corresponding
	arrow labeling}
\label{fig:vertexarrow}
  \end{figure}

See \cref{fig:vertexarrow} for an example of a $\bullet$-labeling of the 
quiver $Q$ from 
\cref{fig:examplestarredquiver}, together with the corresponding 
arrow labeling. 

We now give
a simple characterisation of 
the maps $M: \Arr(Q)\to \R$ that arise from $\bullet$-labelings.   
\begin{definition}\label{def:proper}
 Let $Q$ be a (possibly starred) quiver, and let $\overline {Q}$ be the related quiver where all starred vertices are identified, see \cref{rem:one}. Let $M: \Arr(Q)\to \R$ be an arrow labeling of $Q$. We call $M$ a \emph{$0$-sum arrow labeling} if $\sum_{a\in\pi}\varepsilon_a M(a)=0$ whenever $\pi$ is a closed path in the  underlying graph of $\overline Q$, 
where $\varepsilon_a=1$ 
for arrows $a$ pointing in the direction of the path, and $\varepsilon_a=-1$ for all other arrows. 
We call $M$ \emph{non-trivial} if 
	there is at least one arrow $a$ with $M(a) \neq 0$. We use the notation $\mathbf M_{Q,\R}$ to denote the vector space of $0$-sum arrow labelings, and $\mathbf M_Q$ for its sublattice of $\Z$-valued $0$-sum arrow labelings. 
\end{definition}
We can go from a $0$-sum arrow labeling to a vertex 
labeling of $Q$. Moreover, for strongly connected quivers, there is a simpler description of $0$-sum arrow labelings.

\begin{lemma}
\label{lem:bullet-0sum}
Let $Q$ be a starred quiver. 
If $L$ is a $\bullet$-labeling of $Q$, then 
the corresponding arrow labeling $M_L$ is a $0$-sum arrow labeling. 
Conversely, 
let $M: \Arr(Q)\to \R$ be a $0$-sum arrow labeling. Then 
$M = M_L$ for some $\bullet$-labeling $L:\nV \to \R^n$.

If $Q$ is a strongly connected starred quiver, then  $M: \Arr(Q)\to \R$ is a $0$-sum arrow labeling if and only if  for each oriented path $\pi$ in $Q$ between two starred vertices, we have $\sum_{a\in \pi} M(a)=0$.
\end{lemma}

\begin{proof}
If $L$ is a $\bullet$-labeling of $Q$, then 
it follows from the definitions
that the corresponding arrow labeling $M_L$ is a $0$-sum arrow labeling.

For the converse, let $M$ be a $0$-sum arrow labeling, and let $v\in \nV$. We may assume that $Q$ has a unique starred vertex $\star$ by identifying all starred vertices to one. Since the underlying graph of $Q$ is connected, as we assume throughout, we may now choose an (unoriented) path $p$ from $\star$ to $v$. Let $L(v):=\sum_{a\in p} \varepsilon_a M(a)$ where $\varepsilon_a=1$ if the arrow $a$ is pointing in the direction of the path $p$ and otherwise $\varepsilon_a=-1$. The $0$-sum condition implies that $L(v)$ depends only on the vertex $v\in\nV$ and not on the choice of the path $p$. It is straightforward that $M_L=M$.   

Now suppose $Q$ is a strongly-connected starred quiver with a $0$-sum arrow labeling $M$. Clearly, if $\pi$ is any oriented path from a starred vertex to a starred vertex, then we have $\sum_{a\in\pi}M(a)=0$. (Note that $\pi$ gives rise to a closed oriented path in $\overline Q$.) We now check the converse. Let $v\in\nV$. Since $Q$ is strongly-connected, there exists an \textit{oriented} path $p$ from a starred vertex to $v$. Let us define a vertex labeling $L$ by $L(v):=\sum_{a\in p} M(a)$.
To see that this is well-defined, consider
any other oriented path $p'$ from a (possibly different) starred vertex to $v$. Since $Q$ is strongly-connected,
we also have a path $p''$ from $v$ to a third starred vertex.
Since the concatenations of both $p$ with $p''$ and of  
$p'$ with $p''$ are paths from a starred vertex to a starred vertex, our assumption on $M$ now says that 
	$\sum_{a\in p} M(a) + \sum_{a\in p''} M(a)=0$
	and also 
	$\sum_{a\in p'} M(a) + \sum_{a\in p''} M(a)=0$. This implies that 
$\sum_{a\in p} M(a)=
	\sum_{a\in p'} M(a)$, and therefore $L$ is well-defined. Again it is straightforward that $M=M_L$. It follows from this that $M$ is a $0$-sum arrow labeling as defined in \cref{def:proper}. 
\end{proof}

\begin{remark}\label{rem:isomorphism}
The map given in \cref{lem:bullet-0sum} from the space of vertex coordinates $\R^{\nV}$ to $\mathbf M_{Q,\R}$ 
	(sending a $\bullet$-labeling $L$ to its associated $0$-sum arrow-labeling $M_L$) 
	defines a vector space isomorphism $\Psi:\R^{\nV}\to\mathbf M_{Q,\R}$ that is moreover an integral equivalence. 
This follows directly from its construction in \cref{def:coords} and the construction of the inverse in the proof of \cref{lem:bullet-0sum}. 
\end{remark}

\begin{definition}
Given a $0$-sum arrow labeling $M$ of $Q$, 
we let $M_{\min} = 
\min\{M(a) \ \vert \ a\in \Arr(Q)\}$
be the minimal arrow label occuring in $M$.
\end{definition}

\begin{lemma}\label{lem:face}
Let $Q$ be any starred quiver.
Let $M= M_L$ be a $0$-sum arrow labeling of $Q$ associated
to the $\bullet$-labeling $L\in\R^{\V_\bullet}$.
Let $F(M) = \{a\in \Arr(Q) \ \vert \ M(a)=M_{\min} \}$
be the set of arrows with label $M_{\min}$.
Consider the linear functional on 
$\NP(Q)$ defined by taking the dot product with $L$. This linear functional is minimized precisely
at the face of $\NP(Q)$ containing
vertices 
$\{u_a \ \vert \ a\in F(M)\}$; moreover 
the dot product $L \cdot u_a = M_{\min}$
for all $a\in F(M)$.
\end{lemma}

\begin{proof}
Let us compute the dot product of $L$ with an arbitrary 
vertex $u_a=(u_1,\dots,u_n)$ of $\NP(Q)$.
Suppose that $a:\vv_i \to \vv_j$ joins two normal 
vertices of $Q$.  Then $u_j=1$, $u_i=-1$,
and $u_h=0$ for all other $h$.
Therefore the dot product $L \cdot u_a$  of $L$ with 
$u_a$ is $L(\vv_j)-L(\vv_i) = M(a)$, compare \cref{def:coords}.(1).

Now suppose that $a:\star \to \vv_j$ joins 
a starred vertex to a normal vertex.
Then $u_j=1$ and $u_h=0$ for all other $h$. Therefore  $L\cdot u_a=L(\vv_j)=M(a)$, 
	compare \cref{def:coords}.(2).
Similarly, if $a:\vv_i \to \star$ joins
a normal vertex to a starred vertex,
we have $u_i=-1$ and $u_h=0$ for all other $h$, and
the dot product $L \cdot u_a=-L(\vv_i) = M(a)$, by \cref{def:coords}.(3).

In all cases, the dot product 
$L \cdot u_a$ of $L$ with 
	the vertex $u_a$ is $M(a) \geq M_{\min}$,
with equality attained precisely 
at the vertices 
$\{u_a \ \vert \ a\in F(M)\}$. This is a nonempty set of vertices by construction, and it defines  a face of $\NP(Q)$ because of the linear functional $L$. 
\end{proof}

We will assume throughout the remainder of this section that starred quivers are 
strongly-connected. We will also  identify all starred vertices to a single vertex 
$\star$ as in \cref{r:singlestar}.

\begin{remark}\label{min:neg}
For any $0$-sum arrow labeling $M$ of a strongly-connected starred quiver $Q$,
 we always have $M_{\min} \leq 0$,
since the sum of the arrow labels in 
each path from a starred vertex to a starred vertex is $0$.
Thus if $Q$ is strongly connected and $M$ is non-trivial, then $M_{\min} < 0$.
\end{remark}

\cref{lem:minus1} follows from 
the definition of $0$-sum arrow labeling,
 \Cref{min:neg}, and \Cref{lem:face}.
\begin{lemma}\label{lem:minus1}
Let $M$ be a nontrivial $0$-sum arrow labeling of a strongly connected starred quiver~$Q$.
	Let $M':\Arr(Q)\to \R$ be defined by 
setting $M'(a) = \frac{1}{|M_{\min}|} M(a)$
for all arrows $a \in \Arr(Q)$. Then $M'$
is a $0$-sum arrow labeling
with $M'_{\min} = -1$, and both $M$ and $M'$ are 
	minimized at the same set of vertices of $\NP(Q)$.
\end{lemma}

\begin{definition}
We say that a $0$-sum arrow labeling $M$ of $Q$ is a 
\emph{face arrow-labeling} if 
$M_{\min} = -1$.
\end{definition}

If we want to understand faces of $\NP(Q)$,
then by \Cref{lem:minus1},
we can restrict our attention to face arrow-labelings of $Q$.

\begin{proposition}\label{prop:facet1}
Consider a face  
$F$ of $\NP(Q)$
with vertices $\{u_a | a \in S\}$
for some $S \subseteq \Arr(Q)$.
Then this face is minimized
by a linear functional $L$ such that 
$M:=M_L$ is a face arrow-labeling,
and $S = \{a\in \Arr(Q) \ \vert \ M(a)=-1\}$.
\end{proposition}

\begin{proof}
Since $F$ is a face, 
it is minimized by some linear functional $L:\nV\to \R$.
Letting $r\in \R$ be the minimal value attained, 
we have that $L \cdot u_a = r$ for 
 all $a\in S$,
and $L \cdot u_a > r$ for $a\notin S$.
By \Cref{lem:0interior} we know that 
$0$ lies in the interior of $\NP(Q)$, so $L\cdot 0=0$.
Therefore $r$ must be negative; 
by multiplying all entries of $L$ by $\frac{1}{|r|}$,
we can assume that $L \cdot u_a=-1$ for $a\in S$,
and $L \cdot u_a>-1$ for $a\notin S$.

Now $L$ gives rise to a $0$-sum arrow labeling $M:=M_L$.
Consider a vertex $u_a$ with $a\in S$; suppose
$a: \vv_i \to \vv_j$.  Then $M(a) = L(\vv_j)-L(\vv_i)$
and also $-1 = L \cdot u_a = L \cdot (e_j-e_i)
= L(\vv_j) - L(\vv_i)$, which implies that $M(a)=-1$.
Similarly if $a\in S$ is an arrow connecting a 
starred vertex with a normal vertex, $M(a)=-1$.

For any $u_a$ with $a\notin S$, 
we have $L \cdot u_a > -1$, so the same argument shows 
that $M(a) > -1$.
Therefore $M$ is a face arrow-labeling.
\end{proof}

\begin{definition}\label{def:facet}
Let $Q$ be a strongly connected starred quiver. 
	We say that a face arrow-labeling $M$ of $Q$ is a 
\emph{facet arrow-labeling} if 
the set $F(M) = \{a\in \Arr(Q) \ \vert \ M(a)=-1 \}$
is maximal by inclusion among all face arrow-labelings.

In other words, a facet arrow-labeling is 
an arrow labeling  $M: \Arr(Q)\to \R_{\geq -1}$  such that:
\begin{enumerate}
\item  \label{first}
$M_{\min}=-1$, and 
for each (oriented) path $p$ in $Q$ from a starred vertex to a starred vertex,
we have $\sum_{a\in p} M(a)=0$;
\item  There is no $M': \Arr(Q) \to \R_{\geq -1}$ satisfying
	\eqref{first} 
		such that $F(M')$ properly contains $F(M)$.
\end{enumerate}
\end{definition}

\cref{fig:FacetLabels} shows the four facet arrow-labelings of a particular
starred quiver.
And the arrow labeling from the 
right of \cref{fig:vertexarrow2} is a facet arrow-labeling.

\begin{figure}[h]
\setlength{\unitlength}{1.2mm}
	\begin{picture}(25,25)(-12,0)
	 \put(-9.5,9.5){$-1$}
	 \put(9,9.5){$1$}
	 \put(-3.5,12){$-1$}
	 \put(-1.5,5){$2$}
	\color{red}
	 \put(0,16){$\bullet$}
		\qbezier(0.5,16.5)(14,28)(14,16)
	 \qbezier(14,16)(14,10)(0.5,0.5)
	 \qbezier(0.5,16.5)(-13,28)(-13,16)
	 \qbezier(-13,16)(-13,10)(0.5,0.5)
	 \put(0,0){$\star$}
	 \put(-10,9.5){{\vector(-2,2.2){1}}}
	 \put(11,9.5){{\vector(-2,-2){1}}}
	 \put(0.5,16.5){{\vector(0,-1){7.5}}}
	 \put(0.5,8.5){{\vector(0,-1){7.5}}}
         \put(0,8){$\bullet$}
		\color{black}
 \end{picture}\hspace{.5cm}
	\begin{picture}(25,20)(-12,0)
	\color{red}
	 \put(0,16){$\bullet$}
		\qbezier(0.5,16.5)(14,28)(14,16)
	 \qbezier(14,16)(14,10)(0.5,0.5)
	 \qbezier(0.5,16.5)(-13,28)(-13,16)
	 \qbezier(-13,16)(-13,10)(0.5,0.5)
	 \put(0,0){$\star$}
	 \put(-10,9.5){{\vector(-2,2.2){1}}}
	 \put(11,9.5){{\vector(-2,-2){1}}}
	 \put(0.5,16.5){{\vector(0,-1){7.5}}}
	 \put(0.5,8.5){{\vector(0,-1){7.5}}}
         \put(0,8){$\bullet$}
		\color{black}
	 \put(-9.5,9.5){$-1$}
	 \put(9,9.5){$1$}
	 \put(-1.5,12){$2$}
	 \put(-3.5,5){$-1$}
 \end{picture}\hspace{.5cm}
	\begin{picture}(25,20)(-12,0)
	 \put(-9.5,9.5){$1$}
	 \put(7,9.5){$-1$}
	 \put(-3.5,12){$-1$}
	 \put(-1.5,5){$0$}
	\color{red}
	 \put(0,16){$\bullet$}
		\qbezier(0.5,16.5)(14,28)(14,16)
	 \qbezier(14,16)(14,10)(0.5,0.5)
	 \qbezier(0.5,16.5)(-13,28)(-13,16)
	 \qbezier(-13,16)(-13,10)(0.5,0.5)
	 \put(0,0){$\star$}
	 \put(-10,9.5){{\vector(-2,2.2){1}}}
	 \put(11,9.5){{\vector(-2,-2){1}}}
	 \put(0.5,16.5){{\vector(0,-1){7.5}}}
	 \put(0.5,8.5){{\vector(0,-1){7.5}}}
         \put(0,8){$\bullet$}
		\color{black}
 \end{picture}\hspace{.5cm}
	\begin{picture}(25,20)(-12,0)
	 \put(-9.5,9.5){$1$}
	 \put(7,9.5){$-1$}
	 \put(11,9.5){{\vector(-2,-2){1}}}
	 \put(0.5,16.5){{\vector(0,-1){7.5}}}
	 \put(-1.5,12){$0$}
	 \put(0.5,8.5){{\vector(0,-1){7.5}}}
	 \put(-3.5,5){$-1$}
	\color{red}
	 \put(0,16){$\bullet$}
		\qbezier(0.5,16.5)(14,28)(14,16)
	 \qbezier(14,16)(14,10)(0.5,0.5)
	 \qbezier(0.5,16.5)(-13,28)(-13,16)
	 \qbezier(-13,16)(-13,10)(0.5,0.5)
	 \put(0,0){$\star$}
	 \put(-10,9.5){{\vector(-2,2.2){1}}}
	 \put(11,9.5){{\vector(-2,-2){1}}}
	 \put(0.5,16.5){{\vector(0,-1){7.5}}}
	 \put(0.5,8.5){{\vector(0,-1){7.5}}}
         \put(0,8){$\bullet$}
		\color{black}
 \end{picture}
	\caption{A starred quiver and its
	four facet arrow-labelings.}
\label{fig:FacetLabels}
  \end{figure}
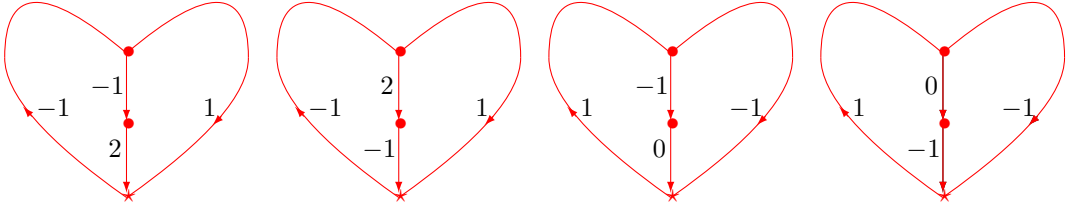

\begin{corollary}\label{cor:facets}
Every facet $F$ of $\NP(Q)$ is 
defined by (minimized by) a linear functional $L$ such that 
$M_L$ is a facet arrow-labeling; 
moreover, the vertices of $F$ are precisely the 
vertices $\{u_a \ \vert \ M_L(a) = -1\}$.
\end{corollary}
\begin{proof}
The fact that $M_L$ is a face arrow-labeling and that $F$ has vertices
 $\{u_a \ \vert \ M_L(a) = -1\}$ comes from 
\Cref{prop:facet1}. Additionally, $M_L$ must be a facet arrow-labeling, 
i.e. the set $F(M_L) = \{a\in \Arr(Q) \ \vert \ M_L(a)=-1 \}$
must be maximal by inclusion among all face arrow-labelings, because otherwise there would 
be another face arrow-labeling $M$ defining a face of $\NP(Q)$
whose vertices are a superset of the vertices of $F$.
\end{proof}

Our next goal is to show that facet arrow-labelings 
 $M:\Arr(Q) \to \R_{\geq -1}$ take on only integer values.  
 This requires a little preparation.

\begin{definition}\label{def:component}
Given a facet arrow-labeling $M:\Arr(Q) \to \R_{\geq -1}$,
we define a \emph{facet component} of $Q$ (with respect to $M$)
to be a connected subquiver $C$ of $Q$ in which 
$M(a)=-1$ for all $a\in \Arr(C)$, and such that $C$ is maximal 
by inclusion with this property.  
\end{definition}
Note that every normal vertex 
and every starred vertex of $Q$ is part of some facet component
of $Q$ (with respect to $M$);
note also that a facet component can be made up of only one single vertex
and no arrows.

\begin{lemma}\label{lem:containstar}
Let $M:\Arr(Q) \to \R_{\geq -1}$ be a facet arrow-labeling.
Then each facet component of $Q$ (with respect to $M$)
contains some starred vertex.
\end{lemma}
\begin{proof}
Let $L$ be the $\bullet$-labeling of $Q$ such that $M_L = M$.
Suppose that there is a facet component $C$ with no starred vertex.
Consider all arrows $a$ of $Q$ which are \emph{adjacent} to $C$
(i.e. such that $a \notin \Arr(C)$ but $a$ shares a vertex with 
$C$), and choose the arrow $\tilde{a}$ such that 
$M(\tilde{a})$ is minimal.
We will use $\tilde{a}$ to produce a new $\bullet$-labeling $L'$.

Let $\ell:=M(\tilde{a})$. Clearly $\ell > -1$.
Suppose without loss of generality that 
$\tilde{a}$ points towards $C$ (the case where $\tilde{a}$
points away from $C$ is similar).
Then for each vertex $v$ of $C$, 
set $L'(v) = L(v) - (\ell+1)$.
For all other vertices $v$ of $Q$, we set $L'(v) = L(v)$.
Now let us consider how $M':=M_{L'}$ differs from $M=M_L$.
Clearly the only arrows $a$ in which $M'(a) \neq M(a)$
are those arrows $a$ which are adjacent to $C$.
Among these arrows:
\begin{itemize}
	\item if $a$ points towards $C$, then $M'(a) = M(a)-(\ell+1)$.
	\item if $a$ points away from $C$, then $M'(a) = M(a)+(\ell+1)$.
\end{itemize}
By our assumptions, each arrow $a$ adjacent to $C$ satisfies
$M(a) \geq \ell > -1$, and hence $M'(a) \geq -1$.
Moreover, $M'(\tilde{a}) = -1$.

In summary,  we have produced a new  $M'$ whose
set of arrows labeled by $-1$ strictly contains the set of arrows labeled 
by $-1$ in $M$. So $M$ is not a facet arrow-labeling, which is a contradiction.
\end{proof}

\begin{proposition}\label{prop:integral}
Let $M:\Arr(Q) \to \R_{\geq -1}$
be a facet arrow-labeling. 
Then  $M$ must be integral: 
for each $a\in \Arr(Q)$, $M(a)\in \Z$.
\end{proposition}

\begin{proof}
Let $M:\Arr(Q) \to \R_{\geq -1}$
be a facet arrow-labeling.  Write $M = M_L$ where $L$ is the corresponding
$\bullet$-labeling.
Let $C$ be a facet component
of $Q$
(with respect to $M$).
By \Cref{lem:containstar}, $C$
contains a starred vertex.
Now, using rules (1), (2), (3) of 
\Cref{def:coords}, and the fact that each arrow $a\in \Arr(C)$
satisfies $M(a)=-1$, we see that 
for each vertex $v$ of $C$, we must have $L(v)\in \Z$.
Since every vertex of $Q$ lies in some facet component,
we must have $L(v) \in \Z$ for \emph{all} vertices
of $Q$.   Therefore $M_L(a) \in \Z$ for all arrows of $Q$.
\end{proof}

The following is the main result of this section.
While this result is stated for strongly-connected starred quivers, 
by \cref{lem:project}, it also applies to strongly-connected quivers.
\begin{theorem}\label{t:reflexive}
Let $Q$ be a strongly-connected starred quiver with 
{normal vertices}
$\nV=\{\vv_1,\dots,\vv_n\}$, and consider the 
corresponding polytope $\NP(Q)$.
Then the facets of $\NP(Q)$ are in bijection with the $\bullet$-labelings 
$L: \nV \to \Z_{\geq -1}$ such that $M_L$ is a facet arrow-labeling.  Moreover, the facet inequality
is $$\{L(\vv_1) x_1 + L(\vv_2) x_2 + \dots + L(\vv_n) x_n \geq -1\},$$
and $\NP(Q)$ is cut out by the union of these inqualities.
Finally, 
 the polytope $\NP(Q)$ is reflexive.
\end{theorem}
\begin{figure}[h]
\setlength{\unitlength}{1.3mm}
\begin{center}
 \begin{picture}(50,14)
         \put(0,14){$\star_1$}
	 \put(8,6){$\vv_1$}
	 \put(9,8){$-1$}
         \put(17,6){$\vv_2$}
         \put(18,8){$-2$}
          \put(26,6){$\vv_3$}
          \put(27,8){$-3$}
         \put(35,6){$\vv_4$}
         \put(36,8){$1$}
         \put(44,6){$\star_2$}
         \put(8,-2){$\vv_5$}
         \put(9,0){$-2$}
         \put(17,-2){$\vv_6$}
         \put(18.5,0){$-3$}
	 \put(26,-2){$\star_3$}
         \put(2,13){{\vector(1,-1){6}}}
         \put(11,6.5){{\vector(1,0){5.5}}}
         \put(20.5, 6.5){{\vector(1,0){5.5}}}
         \put(10.5,-2){{\vector(1,0){5.5}}}
         \put(8.5,5){{\vector(0,-1){5.5}}}
         \put(18,5){{\vector(0,-1){5.5}}}
         \put(20.5,-2){{\vector(1,0){5.5}}}
         \put(29.5, 6.5){{\vector(1,0){5.5}}}
         \put(38.5, 6.5){{\vector(1,0){5.5}}}
 \end{picture}\hspace{.5cm}
 \begin{picture}(50,14)
         \put(0,14){$\star_1$}
	 \put(8,6){$\vv_1$}
	 \put(5,10){$-1$}
	 \put(11,7){$-1$}
         \put(17,6){$\vv_2$}
         \put(22,7){$-1$}
          \put(26,6){$\vv_3$}
          \put(30,7){$4$}
         \put(35,6){$\vv_4$}
         \put(38,7){$-1$}
         \put(44,6){$\star_2$}
         \put(8,-2){$\vv_5$}
         \put(5,2){$-1$}
         \put(8,-2){$\vv_5$}
         \put(14.5,2){$-1$}
         \put(17,-2){$\vv_6$}
         \put(11,-1.5){$-1$}
         \put(22,-1.5){$3$}
	 \put(26,-2){$\star_3$}
         \put(2,13){{\vector(1,-1){6}}}
         \put(11,6.5){{\vector(1,0){5.5}}}
         \put(20.5, 6.5){{\vector(1,0){5.5}}}
         \put(10.5,-2){{\vector(1,0){5.5}}}
         \put(8.5,5){{\vector(0,-1){5.5}}}
         \put(18,5){{\vector(0,-1){5.5}}}
         \put(20.5,-2){{\vector(1,0){5.5}}}
         \put(29.5, 6.5){{\vector(1,0){5.5}}}
         \put(38.5, 6.5){{\vector(1,0){5.5}}}
 \end{picture}
\end{center}
	\caption{A $\bullet$-labeling and the corresponding
	arrow labeling}
\label{fig:vertexarrow2}
 \end{figure}
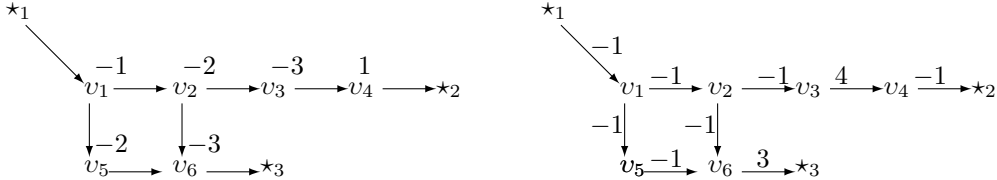

\begin{example}
Consider again the starred quiver from 
	\cref{fig:examplestarredquiver}.
The polytope defined there has facets 
	given in 
	\cref{tab:inequalities}, where e.g. the first 
	row in the table encodes the facet inequality
	$$1+3x_1 + 2x_2 + 2x_3 + x_4+2x_5 + x_6 \geq 0$$
	(which corresponds to the $\bullet$-labeling and arrow labeling in 
	\cref{fig:vertexarrow})
	and the fourth row in the table encodes the facet inequality
	$$1-x_1 - 2x_2 -3x_3 + x_4-2x_5 -3 x_6 \geq 0$$
	(which corresponds to 
the $\bullet$-labeling and arrow labeling 
	in 
	\cref{fig:vertexarrow2}).
\begin{center}
\begin{table}[h]
\begin{tabular}{| p{1cm}| p{1cm} | p{1cm} | p{1cm} | p{1cm} | p{1cm} | p{1cm}| }
\hline
	  & $x_1$ & $x_2$ & $x_3$ & $x_4$ & $x_5$ & $x_6$ \\ \hline
	1 & 3 & 2 & 2 & 1 & 2 & 1\\ \hline
	1 & 3 & 2 & 1 & 0 & 2 & 1\\ \hline
	1 & 3 & 2 & 1 & 1 & 2 & 1\\ \hline
	1 & -1 & -2 & -3 & 1 & -2 & -3 \\ \hline
	1 & -1 & -2 & -3 & 1 & 2 & 1 \\ \hline
	1 & -1 & 2 & 1 & 1 & 2 & 1 \\ \hline
	1 & -1 & 2 & 1 & 1 & -2 & 1 \\ \hline
	1 & -1 & -2 & -3 & 1 & -2 & 1 \\ \hline
	1 & -1 & -2 & 2 & 1 & -2 & 1 \\ \hline
	1 & -1 & 2 & 2 & 1 & -2 & 1 \\ \hline
	1 & -1 & 2 & 2 & 1 & 2 & 1 \\ \hline
	1 & -1 & -2 & 2 & 1 & 2 & 1 \\ \hline
	1 & -1 & -2 & 2 & 1 & -2 & -3 \\ \hline
	1 & -1 & -2 & -3 & -4 & -2 & 1 \\ \hline
	1 & -1 & 2 & 1 & 0 & -2 & 1 \\ \hline
	1 & -1 & 2 & 1 & 0 & 2 & 1 \\ \hline
	1 & -1 & -2 & -3 & -4 & 2 & 1 \\ \hline
	1 & -1 & -2 & -3 & -4 & -2 & -3 \\ \hline
			\end{tabular}
			\caption{\label{tab:inequalities}The inequalities defining  the polytope
			$\NP(Q)$ where $Q$ is the quiver from
	\cref{fig:examplestarredquiver} and 
\cref{fig:vertexarrow2}}
		\end{table}
	\end{center}
\end{example}

\begin{remark}\label{r:TU}
The statement that $\NP(Q)$ is reflexive also appears in \cite[Proposition 4.2]{Higashitani}, which 
additionally says that $\NP(Q)$ is terminal.
\cite{Higashitani} does not provide a proof but asserts it can be done
similarly to  \cite[Proposition 3.2]{Matsui} (which proved the reflexive
property in the case of symmetric directed graphs)
and  \cite[Lemma 1.2]{Ohsugi}
(which discusses when the origin lies in the interior). 
Another approach to reflexivity is to 
observe that the matrix whose rows encode
the coordinates of the vertices of $\NP(Q)$ is a
	\textit{totally unimodular} matrix (an integer matrix all of whose minors are $\pm 1$ or $0$) \cite{GH}, see also \cite[Theorem 29]{Kaibel}.
By a version of the Hoffman-Kruskal Theorem \cite{HoffmanKruskal, Schrijver},
	this implies that the polar dual of $\NP(Q)$, if it is nonempty, is integral, 
\end{remark}

\begin{proof}[Proof of 
	\cref{t:reflexive}]
By \Cref{cor:facets}, each facet $F$ of $\NP(Q)$ is minimized by a linear functional
$L:\nV\to \R_{\geq -1}$ such that $M_L$ is a facet arrow-labeling, and the 
vertices of $F$ are precisely the vertices $\{u_a \ \vert \ M_L(a)=-1\}$.
Let $\mathcal{L}$ be the set of linear functionals corresponding to facets of $\NP(Q)$.
Note that for $L\in \mathcal{L}$,
$L \cdot u_a = M_L(a)$, so the inequality cutting out this facet is 
$L \cdot (x_1,\dots,x_n) \geq -1$, i.e. 
$\{L(\vv_1) x_1 + L(\vv_2) x_2 + \dots + L(\vv_n) x_n \geq -1\}.$
And by \Cref{prop:integral}, each $L(\vv_i)\in \Z$.

Therefore, using the correspondence  between
the $i$-dimensional faces of a polytope and the
$(n-i-1)$-dimensional faces of its polar dual, we
see that the vertices of the polar dual $\NP(Q)^*$ are 
the points $\{(L(\vv_1),\dots,L(\vv_n)) \ \vert \ L\in \mathcal{L}\}$.
In particular, they are integral.
We already showed in 
\Cref{lem:0interior}
 that the origin is in the interior of $\NP(Q)$,
so $\NP(Q)$ is reflexive.
\end{proof}

\section{The planar setting: relation to flow polytopes}\label{sec:planar}

Let $Q$ be a connected quiver with no oriented cycles.  
One can associate to each quiver a reflexive polytope called 
a \emph{flow polytope}, as shown in \cite{Hille}.
Flow polytopes are an interesting class of polytopes that 
are closely connected to toric geometry, as we explain in \cref{sec:toric}, as well as
 representation theory \cite{Baldoni} and 
gentle algebras \cite{gentle}.
There has also been a great deal of work 
on their volumes \cite{Baldoni, Meszaros2, MM, CKM, GHMY}
and faces \cite{AltmannvanStraten}. 
In this section we show that if $Q$ 
is a strongly-connected  quiver 
with a unique starred vertex which is planar, 
then 
its root polytope $\NP(Q)$ is polar dual to the flow 
polytope of the dual quiver.
This gives a new proof that the flow polytope associated to a plane acyclic quiver is reflexive. 

\begin{definition}\label{def:flow}
Let $Q=(Q_0, Q_1)$ be a connected quiver which is 
\emph{acyclic}, that is, it has no  oriented cycles.  
We denote 
elements of $\R^{Q_1}$ by $(r_a)_{a\in Q_1}$.
We let $V_Q$ 
	be the subspace of $\R^{Q_1}$
	obtained by imposing for each vertex $q\in Q_0$ the 
	relation 
	\begin{equation}\label{e:VQ}\sum_{a \rightarrow q} r_a = \sum_{b \leftarrow q} r_b,
	\end{equation}
where the first sum is over all arrows with target $q$,
and the second sum is over all arrows with source $q$.
Then the \emph{flow polytope} $\Fl_Q$
is defined to be $$\Fl_Q = \{(r_a)_{a \in Q_1} \ \vert \ r_a \geq -1\} \subset V_Q.$$
\end{definition}
The flow polytope as defined above is also known as the canonical weight flow polytope. See \cref{rem:defflow} for a comparison to the original definition of \cite{Hille}.

A \emph{plane quiver} is a quiver which is properly embedded into the plane $\R^2$, 
that is, no two arrows cross each other.

\begin{definition}\label{def:dual}
Let $Q=(Q_0,Q_1)$ be an acyclic connected plane quiver.
  Let 
	$Q_2$ be the set of bounded regions of 
	$\R^2 \setminus Q$.
	We construct the \emph{dual starred quiver} 
	$Q^{\vee}=(Q_{\bullet}^{\vee} \sqcup \{\star\}, Q_1^{\vee})$
by placing one normal vertex in each bounded region of $\R^2\setminus Q$
	so that $Q_{\bullet}^{\vee} \cong Q_2$,
	and placing a starred vertex $\star$ in the unbounded region
	of $\R^2\setminus Q$.
	For each arrow $q\in Q_1$ separating two regions of 
	$\R^2 \setminus Q$, we also place a dual arrow 
	$a^{\vee}\in Q_1^{\vee}$ which crosses the original arrow 
	$a$ from ``left to right,'' see \cref{fig:dual}.
\end{definition}
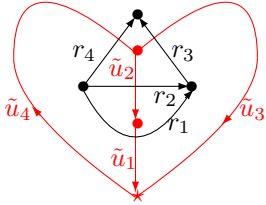
\begin{figure}[h]
\setlength{\unitlength}{1.2mm}\begin{center}
	\begin{picture}(25,20)(-12,0)
	\color{red}
	 \put(0,16){$\bullet$}
		\qbezier(0.5,16.5)(14,28)(14,16)
	 \qbezier(14,16)(14,10)(0.5,0.5)
	 \qbezier(0.5,16.5)(-13,28)(-13,16)
	 \qbezier(-13,16)(-13,10)(0.5,0.5)
	 \put(0,0){$\star$}
	 \put(-10,9.5){{\vector(-2,2.2){1}}}
	 \put(11,9.5){{\vector(-2,-2){1}}}
	 \put(0.5,16.5){{\vector(0,-1){7.5}}}
	 \put(0.5,8.5){{\vector(0,-1){7.5}}}
         \put(0,8){$\bullet$}
	 \put(-14,9.5){$\uu_4$}
	 \put(12,9.5){$\uu_3$}
	 \put(-2.5,14){$\uu_2$}
	 \put(-2.3,4){$\uu_1$}
		  \color{black}
		\put(0,20){$\bullet$}
		\put(-6,12){$\bullet$}
		\put(6,12){$\bullet$}
		\put(-5.5,12.5){{\vector(6,8){6}}}
		\put(6.8,12.5){{\vector(-6,8){6}}}
		\put(-5.5,12.8){{\vector(1,0){12}}}
		\qbezier(-5.5,12.8)(0,2)(6.5,12)
		\put(6,11.5){{\vector(1,1){1}}}
	 \put(-6.5,16){$r_4$}
	 \put(4.5,16){$r_3$}
	 \put(2.5,11){$r_2$}
	 \put(4,8){$r_1$}
 \end{picture}\hspace{.5cm}\end{center}
	\caption{A plane acyclic quiver $Q$ (in black, with arrows
	labeled $r_i$) and the dual starred quiver
	$Q^{\vee}$ (in red, with arrows labeled $\uu_i$). 
	In $V_Q$ 
	(see \cref{def:flow})
	an element
	$(r_1,r_2,r_3,r_4)$ satisfies $r_1+r_2=r_3$, 
	$r_1+r_2+r_4=0$, and $r_3+r_4=0$, with one relation for each vertex
	of $Q$.
	In $\mathbf N_{Q^{\vee}, \R}$ 
	(see \cref{def:NNP})
	the elements $\uu_1,\uu_2,\uu_3,\uu_4$ 
	satisfy $\uu_1+\uu_2-\uu_3=0$, $\uu_1+\uu_2+\uu_4=0$, and 
	$\uu_3+\uu_4=0$,
	with one relation for each cycle of $Q^{\vee}$}
	\label{fig:dual}
\end{figure}

The following lemma is easy to check and well-known.
	\begin{lemma}\label{lem:bijection}
The map in \cref{def:dual} gives a bijection between 
acyclic connected plane quivers and 
		strongly-connected
		plane starred  quivers which have a unique starred 
vertex. 
\end{lemma}

The following is the main result of this section.
\begin{theorem}\label{thm:dual}
Let $Q=(Q_0,Q_1)$ be a plane acyclic connected quiver
	and let $Q^{\vee}$ be the dual starred quiver, as in 
	\cref{def:dual}. 
	Then the root polytope $\NP(Q^{\vee})$ 
	is integrally equivalent to the polar dual of the flow polytope
	$\Fl_Q$.
\end{theorem}

\begin{corollary}
Let $Q$ be a connected plane acyclic quiver.  Then
$\Fl_Q$ is reflexive. 
\end{corollary}
\begin{proof}
	This follows from \cref{thm:dual} and 
	\cref{t:reflexive}.
\end{proof}

The corollary reproves in the plane case the result of
 \cite{Hille} that (canonical weight) flow polytopes are reflexive. 

Before proving \cref{thm:dual}, we start by giving a description of 
the root polytope of a starred quiver in terms of arrow coordinates.

\begin{definition}\label{def:NNP}
Let $Q$ be a starred quiver with vertices 
$\nV \sqcup \V_{\star}$ 
(where $\nV=\{\vv_1,\dots,\vv_n\}$ 
	and $\V_{\star}=\{\star_1,\dots,\star_{\ell}\}$),
and arrows $Q_1=\Arr(Q)$. 
	Let 
	$\mathbf N_{Q,\R} = 
	\R^{Q_1}/U$ where, in terms of the standard basis $\{\delta_a\}$ of $\R^{Q_1}$,
\begin{equation}\label{eq:cycle}
U=\langle	\sum_{a\in \pi} \varepsilon_a \delta_a\mid \pi \text{ an undirected path that is closed, or begins and ends in a starred vertex} \rangle_\R.
\end{equation} 
 Here 
$\varepsilon_a=1$ for arrows $a\in Q$ pointing in the direction of the path $\pi$ and $-1$ for all other arrows.  

	Note that $\mathbf N_{Q,\R}$ is dual to the vector space of $0$-sum arrow-labelings $\mathbf M_{Q,\R}$ as constructed in \cref{def:proper}. 
	In particular, the dual of the $\Z$-lattice $\mathbf M_Q\subset\mathbf M_{Q,\R}$ is a $\Z$-lattice in $\mathbf N_{Q,\R}$ that we may  denote by $\mathbf N_Q$.
	This gives $\mathbf N_{Q,\R}$ an integral structure.

	For each $a\in Q_1$, let $\uu_a$ denote the image of $\delta_a$ in $\mathbf N_{Q,\R} = \R^{Q_1}/U$. We define 
	$$\NNP(Q):= \conv(\{\uu_a \ \vert \ a\in Q_1\}) \subseteq \mathbf N_{Q,\R}.$$
\end{definition}
Compare the above definition with \cref{def:NP}.

\begin{example}
	If we consider the `dual starred quiver' $Q^{\vee}$
	shown in \cref{fig:dual}, 
	then the elements $\uu_1,\uu_2,\uu_3,\uu_4\in 
	\mathbf N_{Q^{\vee},\R}$ 
	associated to the arrows 
	in $Q^\vee_1$  
	satisfy the relations $\uu_1+\uu_2-\uu_3=0$, $\uu_1+\uu_2+\uu_4=0$, and $\uu_3+\uu_4=0$.\end{example}

	\begin{proposition}\label{prop:NNP}
Let $Q$ be a starred quiver as in \cref{def:NNP}.  
The lattice polytopes $\NNP(Q)$ and $\NP(Q)$ are integrally equivalent. 
\end{proposition}
\begin{proof}  
Recall the isomorphism
	$\Psi:\R^{\nV}\to\mathbf M_{Q,\R}$ 
	from \cref{rem:isomorphism}
	sending a $\bullet$-labeling $L$ to its associated $0$-sum arrow-labeling $M_L$. 
Let us write $(\R^{\nV})^*=\R^n$; this is the vector space containing $\NP(Q)$.  Recall that we paired 
	$L\in\R^{\nV}$ (viewed as an element of $\R^n$) 
	with $u_a\in\R^n$ using the dot product, $\langle L,u_a\rangle=L\cdot u_a$, and we had $L\cdot u_a=M_L(a)$.  Thus  we have
\[\langle L,\Psi^*(\tilde u_a)\rangle=\langle\Psi(L),\tilde u_a\rangle=\langle M_L,\tilde u_a\rangle=M_L(a)=\langle L, u_a\rangle,
\]
for any $L\in\R^{\nV}$. The dual isomorphism $\Psi^*:\mathbf N_{Q,\R}\to \R^n$ therefore  sends the vertex $\tilde u_a$ of $\NNP(Q)$ to the vertex $u_a$ of $\NP(Q)$. Thus it defines an integral equivalence between $\NNP(Q)$ and  $\NP(Q)$.
\end{proof}

\begin{remark}
For any strongly-connected starred quiver $Q$, the polar dual of the `root polytope' $\NNP(Q)\subset \mathbf N_{Q,\R}$ is precisely the convex hull of the set of facet arrow-labelings of $Q$ (which are indeed lattice points in the dual space $\mathbf M_{Q,\R}$). This is the translation of \cref{t:reflexive} from $\NP(Q)$ to $\NNP(Q)$.  
\end{remark}

\begin{proof}[Proof of \cref{thm:dual}]
	We know from \cref{prop:NNP} that 
	$\NP(Q^{\vee})$ is integrally equivalent to the polytope 
	$\PP:=
	\NNP(Q^{\vee})
	=\conv(\{\uu_{\bar a} \ \vert \ \bar a\in Q_1^{\vee} \})$ in $\mathbf N_{Q^{\vee}, \R}=\mathbf N_{Q,\R}$.
	The polar dual 
	$\PP^*$ of $\PP$ is  
	$$\{r\in (\mathbf N_{Q^{\vee},\R})^*  \mid \langle r, x\rangle \geq -1	\text{ for all }x\in \PP\} = 	\{r\in \mathbf M_{Q^{\vee},\R}  \mid \langle r, x\rangle \geq -1	\text{ for each vertex }x \text{ of }\PP\}.$$ 
Thus we get one inequality of $\PP^*$ for each vertex $\uu_{\bar a}$ of $\PP$, and this inequality is $\langle r,\tilde u_{\bar a}\rangle \geq -1$.  In other words the elements of $\PP^*$ are $0$-sum arrow labelings $(r_{\bar a})_{\bar a\in Q_1^{\vee}}$ satisfying	$r_{\bar a} \geq -1$ for all arrows $\bar a\in Q^\vee_1$.

Let us refer to a closed path in the underlying graph of $Q^{\vee}$ that bounds a connected region of $\R^2\setminus Q^\vee$ as a `minimal' closed path for $Q^\vee$. Then the $0$-sum conditions for $(r_{\bar a})_{\bar a\in Q^\vee_1}$ 
coming from such minimal closed paths generate all of the $0$-sum conditions of $\mathbf M_{Q^\vee,\R}$ from \cref{def:proper}. 
To complete the proof, note that we obtain an isomorphism from $\mathbf M_{Q^\vee,\R}$  to $V_Q$ by the map sending $(r_{\bar a})_{\bar a\in Q^\vee_1}$ to $(r_a)_{a\in Q_1}$ (where  $r_{\bar a}=r_a$ if $a$ is the arrow in $Q_1$ dual to the arrow $\bar a\in Q^\vee_1$). Namely, the $0$-sum conditions for $(r_{\bar a})_{\bar a\in Q^\vee_1}$  
coming from minimal closed paths translate precisely to the relations \eqref{e:VQ} 
for $(r_a)_{a\in Q_1}$ that define $V_Q$. This isomorphism is clearly an integral equivalence, and moreover  
the inequalities $r_{\bar a}\ge -1$ describing $\PP^*$ translate to the inequalities $r_a\ge -1$ defining $\Fl_Q$. Thus $\Fl_Q$ is integrally equivalent to the polar dual of $\NP(Q^\vee)$.
\end{proof}

\begin{remark} \label{rem:defflow}
The original definition of flow polytope (associated to the 
\emph{canonical weight} $\theta = \theta^c$) from \cite{Hille} looks slightly different
but is easily seen to be equivalent to \Cref{def:flow}.
Following \cite{Hille}, 
let $Q=(Q_0,Q_1)$ be a connected quiver with no oriented cycles, and 
for each vertex $q$,  set
$$\theta(q):=\#\{ \text{arrows with source }q\}-\#\{\text{arrows with 
target }q\}.$$
We denote elements of $\R^{Q_1}$ by $(R_a)_{a\in Q_1}$.
Let $V_Q$ be the subspace  of $\R^{Q_1}$
obtained by imposing for each vertex $q\in Q_0$ the 
	relation $$\theta(q) + \sum_{a \rightarrow q} R_a = \sum_{b \leftarrow q} R_b.$$
We then define the 
flow polytope $\Delta(Q)$ as 
$$\Delta(Q)= \{(R_a)_{a\in Q_1} \ \vert \  R_a \geq 0\} - \mathbf{1} \subset V_Q,$$
	where $\mathbf{1}$ is the all-one vector.
	If we set $r_a = R_a-1$, we see that $\Fl_Q = \Delta(Q)$.
\end{remark}

\begin{remark} 
\cref{thm:dual} does not extend to the non-planar case; 
it fails already for quivers obtained from the complete bipartite
graph $K_{3,3}$, where the polar dual of the flow polytope is no longer a root polytope. 
\end{remark}

\section{The ranked poset setting: relation to (marked) order polytopes}
\label{sec:markedorder}

In this section we specialize to the setting where our starred quiver $Q$
comes from a poset.  We will prove \cref{mainthm:poset} and  
\cref{thm:D}, relating root polytopes to marked order polytopes and order polytopes.

\subsection{Background on order polytopes and marked order polytopes}

\emph{Order polytopes} associated to finite posets were 
investigated in 
 \cite{order}.
Given a poset $P=\{\vv_1,\dots,\vv_n\}$ (where we identify $P$ with its 
set of points), we let $\R^P$ denote the set of 
all functions $f: P\to \R$.
	Recall that $\hat{P}$ denotes the \emph{bounded extension} of $P$, obtained
by adjoining a new minimum element $\hat{0}$ and a new maximum element
$\hat{1}$.  

\begin{definition}
	The \emph{order polytope} $\mathcal{O}(P)$ of the poset $P$
	is the subset of $\R^P$ defined by the conditions 
	\begin{align*}
		0 \leq f(\vv_i) \leq 1 \qquad & \text{ for all }\vv_i\in P\\
		f(\vv_i) \leq f(\vv_j) \qquad & \text{ if }\vv_i\leq \vv_j \text{ in }P.
	\end{align*}
	
	Equivalently, we can use $\hat{P}$ and 
	define the polytope $\hat{\mathcal{O}}(P)$ to be the
	set of functions $g\in \R^{\hat{P}}$ satisfying
	\begin{align*}
		&g(\hat{0}) = 0 \text{ and }g(\hat{1}) = 1\\
		&g(\vv_i) \leq g(\vv_j)  \text{ if }\vv_i\leq \vv_j \text{ in }\hat{P}.
	\end{align*}
	The linear map $\hat{\mathcal{O}}(P) \to \mathcal{O}(P)$ obtained
by projection to $\R^P$ is a bijection.
\end{definition}

Recall that a \emph{(lower) order ideal} or \emph{down-set}
of $P$ is a subset $I$ of $P$ such that if 
$\vv_j\in I$ and $\vv_i\leq \vv_j$, then $\vv_i\in I$.
And a \emph{filter} (or \emph{dual order ideal} or \emph{up-set})
of $P$ is a subset $I$ of $P$ such that if 
$\vv_i\in I$ and $\vv_j\geq \vv_i$, then $\vv_j\in I$.
Let $\chi_I: P \to \R$ denote the characteristic function of $I$, i.e.
\begin{equation*}
	\chi_I(\vv_i) = \begin{cases}
		1, &\vv_i\in I\\
		0, &\vv_i\notin I
	\end{cases}
\end{equation*}

\begin{proposition}\cite{order}
The vertices of $\OO(P)$ are in bijection with the filters of $P$: 
namely, 
the vertices of $\OO(P)$ are the characteristic functions
$\chi_I$ of filters $I$ of $P$.

The facets of ${\mathcal{O}}(P)$ are in bijection with the cover relations
of $\hat{P}$: namely, if $\vv_i\lessdot \vv_j$ is a cover relation in $\hat{P}$,
then the corresponding facet consists of those $g\in \mathcal{O}(P)$
satisfying $g(\vv_i)=g(\vv_j)$.
\end{proposition}

There is a natural generalization of order polytope which first appeared in 
\cite{marked}.

\begin{definition}\label{def:markedposet}
Let  $P = P_{\bullet} \sqcup P_{\star}$ 
	be a \emph{starred poset} as  in \cref{quiverfromposet}.
	We write $P_{\bullet} = \{\vv_1,\dots,\vv_n\}$
	and $P_{\star} =\{\star_1,\dots,\star_{\ell}\}$. 
A pair $(P,m)$ is called a \emph{marked poset} if
 $m:P_{\star} \to \R$ is an order-preserving map on $P_{\star}$, called a \emph{marking}.
\end{definition}

\begin{definition}\label{def:marked}\cite{marked}
	The \emph{marked order polytope} $\mathcal{O}_m(P)$ of the 
	marked poset 
$(P=P_{\bullet} \sqcup P_{\star}, m)$ 
	is the subset of $\R^{P_{\bullet}}$ defined by the conditions 
	\begin{align*}
	m(\star_i) \leq f(\vv_j) &\quad  \text{ if }\star_i\leq \vv_j \text{ where }\star_i\in P_{\star}, \vv_j\in P_{\bullet}\\
		f(\vv_i) \leq m(\star_j) &\quad \text{ if }\vv_i\leq \star_j \text{ where }
		\vv_i\in P_{\bullet}, 
		\star_j \in P_{\star}, 
		\\
		f(\vv_i) \leq f(\vv_j) &\quad \text{ if }\vv_i\leq \vv_j \text{ where }\vv_i, \vv_j \in P_{\bullet}.
	\end{align*}
\end{definition}

\subsection{The relation to marked order polytopes}

In what follows we will
be particularly interested in
\emph{ranked} and \emph{graded} posets.

\begin{definition}\label{def:rank}
Let $P$ be a finite poset with bounded extension $\hat{P} = P \cup \{\hat{0},\hat{1}\}$.  
We say that a poset $P$ is \emph{ranked} if for each $\vv\in P$,
each maximal chain in $P \cup \hat{0}$ from $\hat{0}$ to $\vv$ has the same length.
And we say that $P$ is \emph{graded} 
if all maximal chains in $\hat{P} = P \cup \hat{0} \cup \hat{1}$ 
from $\hat{0}$ to $\hat{1}$ have the same length.
Given a ranked poset $P$, for $\vv\in P$, we let 
$\rank(\vv)$ denote the length of any maximal chain from $\hat{0}$ to ${\vv}$.
For brevity, we will sometimes denote $\rank(\vv)$ by $R(\vv)$.
\end{definition}
\begin{remark}
Note that every graded poset $P$ is ranked 
but the notion of ranked poset is more general.
\end{remark}

\begin{definition}
Let $P$ be a starred poset as in \cref{quiverfromposet}, and
suppose that $P$ is ranked with rank function $R$.
The rank function  gives rise to a marking
$m:P_{\star} \to \Z$ by letting $m(\star_i)=R(\star_i)$ for all $\star_i
	\in P_{\star}$.
We let $\mathcal{O}_R(P)$ denote the marked order polytope associated to the
marking coming from $R$. 
\end{definition}

\cite{FFP} showed that certain marked order polytopes are {reflexive}.

\begin{proposition}\cite[Proposition 3.1 and Theorem 3.4]{FFP}\label{p:reflexive}
Consider
the marked order polytope $\mathcal{O}_R(P)$ associated to $(P, R)$,
where $P =
	P_{\bullet} \sqcup P_{\star}$ is  ranked  with rank function $R$.\footnote{We note
	that \cite{FFP} used a definition of ranked that is slightly more general than the one
	we use here.  \cite{FFP} defines a ranked poset to be a poset $P$
together with  a rank function $R:P\to \Z$
such that for each covering relation $a \lessdot b$ we have $R(b)=R(a)+1$.  So in particular,
	\cite{FFP} would allow two minimal elements to have different ranks.}
Then $\OO_R(P)$ contains a unique interior lattice point
$\mathbf{u} = (u_\vv)_{\vv\in P_{\bullet}}$, where $u_\vv=R(\vv)$.
Moreover, if we let
$\overline{\OO}_R(P)
:=\OO_R(P)-\mathbf{u}$, then
$\overline{\OO}_R(P)$ is a reflexive polytope.
\end{proposition}

	\begin{figure}
\[ \adjustbox{scale=.75,center}{\begin{tikzcd}
	{\star_1} &&& { 5\, \star\ \ \,} \\
	{v_4\bullet\quad} & {\star_2} && { 4\, \bullet\ \ \,} & { \ \ \,\star\, 4} \\
	{v_3\bullet\quad} & {\ \ \ \bullet v_6} && { 3\, \bullet\ \ \,} & { \ \ \,\bullet\, 3} \\
	{v_2\bullet\quad} & {\ \ \ \bullet v_5} && { 2\, \bullet\ \ \,} & { \ \ \,\bullet\, 2} \\
	{v_1\bullet\quad} &&& { 1\, \bullet\ \ } \\
	{\hat 0\, \star_0\ } &&& 
	\arrow[from=2-1, to=1-1]
	\arrow[from=2-4, to=1-4]
	\arrow[from=3-1, to=2-1]
	\arrow[from=3-2, to=2-2]
	\arrow[from=3-4, to=2-4]
	\arrow[from=3-5, to=2-5]
	\arrow[from=4-1, to=3-1]
	\arrow[from=4-1, to=3-2]
	\arrow[from=4-2, to=3-2]
	\arrow[from=4-4, to=3-4]
	\arrow[from=4-4, to=3-5]
	\arrow[from=4-5, to=3-5]
	\arrow[from=5-1, to=4-1]
	\arrow[from=5-1, to=4-2]
	\arrow[from=5-4, to=4-4]
	\arrow[from=5-4, to=4-5]
	\arrow[from=6-1, to=5-1]
	\arrow[from=6-4, to=5-4]
\end{tikzcd}}\]
	\caption{At left: a marked poset, drawn as a starred quiver.
	At right: the marked poset
	with elements labeled by their ranks\label{fig:markedorder}}
\end{figure}
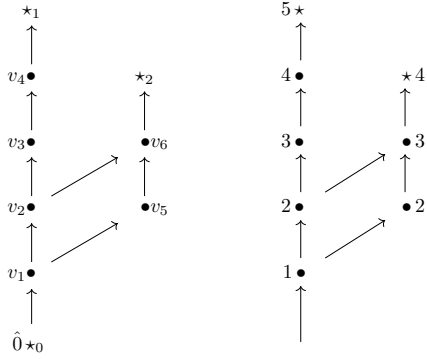

\begin{example}\label{ex:marked}
Let $P$ be the starred poset  
with Hasse diagram shown at the left of \cref{fig:markedorder}, identified
with its corresponding starred quiver $Q$ as in \cref{quiverfromposet}.
The ranks of the elements of $P$ are indicated at the right of 
\cref{fig:markedorder}.  By definition, $\OO_R(P)$
is defined by the inequalities
	$$0 \leq f(\vv_1) \leq f(\vv_2) \leq f(\vv_3) \leq f(\vv_4) \leq 5, \qquad
	f(\vv_2) \leq f(\vv_6), \qquad f(\vv_1) \leq f(\vv_5) \leq f(\vv_6) \leq 4.$$
	Clearly $\OO_R(P)$ has an interior lattice point $\mathbf{u}=(u_1,\dots,u_6)$
	with $u_i =\rank(\vv_i)$, i.e. $u_1=1, u_2=2, u_3=3, u_4=4, u_5=2, u_6=3.$
	Now if we let $F(\vv_i) = f(\vv_i) - \rank(\vv_i)$, or equivalently
	substitute $f(\vv_i)=F(\vv_i)+\rank(\vv_i)$ into the above inequalities,
	 we see that 
	$\overline{\OO}_R(P)$ is cut out by the inequalities 
\begin{align*}
	&  F(\vv_1) \geq -1, &\qquad
	&F(\vv_2)-F(\vv_1) \geq -1, &\qquad
	&F(\vv_3)-F(\vv_2) \geq -1, &\qquad \\
	&F(\vv_4)-F(\vv_3) \geq -1,& \qquad
	&-F(\vv_4) \geq -1, &\qquad 
	&F(\vv_6)-F(\vv_2) \geq -1, & \qquad\\
	& F(\vv_5)-F(\vv_1) \geq -1, &\qquad
	& F(\vv_6)-F(\vv_5) \geq -1, &\qquad
	&-F(\vv_6) \geq -1. &\qquad
\end{align*}
	From these inequalities we see that the polar dual of $\overline{\OO}_R(P)$
	is the polytope with vertices
	$$\{e_1, e_2-e_1, e_3-e_2, e_4-e_3, -e_4, e_6-e_2, e_5-e_1, e_6-e_5, -e_6 \}.$$  This is exactly the root polytope $\NP(Q)$ of our starred quiver $Q$!
\end{example}

More generally, 
the polar dual to $\overline{\OO}_R(P)$ is the 
{root polytope} of the associated starred quiver.

\begin{theorem}\label{prop:NPreflexive}
Let $P=P_{\bullet} \sqcup P_{\star}$ be a starred poset 
	and let $Q:=Q_{(P_{\bullet},P_{\star})}$ be the starred quiver defined in 
\cref{quiverfromposet}.
Suppose further that $P$ is  ranked  with rank function $R$.
Then the polytope
$\overline{\OO}_R(P)$
from \Cref{p:reflexive} is polar dual to the root polytope $\NP(Q)$.
\end{theorem}
One consequence of \cref{prop:NPreflexive} is a new proof that 
$\NP(Q)$ is reflexive in this setting.
\begin{proof}
We know from \cref{p:reflexive} that $\overline{\OO}_R(P) = \OO_R(P)-\mathbf{u}$  is reflexive.  By definition, it
	 has one
	inequality for each cover relation $a\lessdot b$ in the poset
	(where at least one of $a,b$ lies in $P_{\bullet}$):
$$\begin{cases}
	F(a)+R(a) \leq F(b)+R(b) 
	& \text{ if }a,b\in P_{\bullet}\\
	R(a) \leq F(b)+R(b)  & \text{ if }a\in P_{\star} \text{ and }b\in P_{\bullet}\\
	F(a)+R(a) \leq R(b) & \text{ if }a\in P_{\bullet} \text{ and }b\in P_{\star}.
\end{cases}$$
Since our poset is ranked (and hence $R(b)=R(a)+1$ when $a\lessdot b$), these inequalities
become:
$$\begin{cases}
	F(b)-F(a) \geq -1   & \text{ if }a,b\in P_{\bullet}\\
	 F(b) \geq -1  & \text{ if }a\in P_{\star} \text{ and }b\in P_{\bullet}\\
	F(a) \geq -1 & \text{ if }a\in P_{\bullet} \text{ and }b\in P_{\star}.
\end{cases}$$

But then the polar dual has vertices given by those  of the root polytope
associated to the quiver from
\cref{quiverfromposet}, so we are done.
\end{proof}

\cref{prop:NPreflexive} 
allows us to relate the face fan of the root polytope
to the normal fan of the marked order polytope.  
\begin{definition}
Suppose that a convex polytope $\PP$ contains the origin in its
interior.  Then the \emph{face fan}
$\F(\PP)$ is the fan whose 
cones are the cones over the faces of 
$\PP$.
\end{definition}

\begin{definition}
Given a convex polytope $\PP$ in $\R^n$, 
the \emph{inner normal fan} $\NF(\PP)$ of $\PP$
is a polyhedral fan in the dual space
$(\R^n)^*$ whose cones consist of the \emph{normal cone}
$c_F$ to each face $F$ of $\PP$.   That is,
$$\NF(\PP) = \{c_F\}_{F\in \face(\PP)},$$ where 
each normal cone  $c_F$ is defined as the set of 
linear functionals $w$ such that the set of points $x$ in 
$\PP$ that \emph{minimize} $w(x)$ contains $F$,
$$c_F = \{w\in (\R^n)^* \ \vert \ F \subset \argmin_{x\in \PP} w(x)\}.$$
\end{definition}

\begin{corollary}\label{cor:fans}
Let $P$ and $Q$ be the ranked starred poset and its  associated starred quiver from \cref{prop:NPreflexive}.
	Then the face fan $\F_Q:=\F(\NP(Q))$ of 
the root polytope $\NP(Q)$ coincides with the normal fan
of the  marked order
polytope
${\OO}_R(P)$. 
\end{corollary}
\begin{proof} 
This follows immediately from \cref{prop:NPreflexive}, together with the fact that
shifting a polytope does not change its normal fan.
\end{proof}

In the next section we will give an analogue of \cref{cor:fans} but 
we will use the order polytope $\mathcal{O}(P)$
instead of the marked order polytope
$\overline{\OO}_R(P)$.

\subsection{Relation to order polytopes}
In this section we will use the starred quiver 
$Q_{\hat{P}}$ associated to the bounded extension $\hat{P}$ of a finite
poset $P$, as in \cref{rem:quiverfromposet}.  We
will show that the rays of the 
face fan $\F_{Q_{\hat{P}}}$ of the root polytope $\NP(Q_{\hat{P}})$
coincide with the rays of the (inner) normal fan $\NF(\OO(P))$ of 
the order polytope $\OO(P)$, for any finite poset $P$.  
Additionally, if $P$ is a ranked poset,
then $\F_{Q_{\hat{P}}}$ refines $\NF(\OO(P))$: that is,
each maximal cone of $\F_{Q_{\hat{P}}}$ is contained in a maximal cone of 
$\NF(\OO(P))$. 
Finally, if $P$ is a graded poset,
then the two fans agree, namely $\F_{Q_{\hat{P}}}$ equals $\NF(\OO(P))$. The latter result also follows from the observation, found in \cite[Remark 1.6]{HH}, that for a graded poset $P$ the polar dual of $\NP(Q_{\hat P})$ is a dilation of the order polytope $\OO(P)$ up to shift.

\begin{lemma}\label{lem:rays}
Let $P$ be an arbitrary finite poset.
Then the rays of the 
	face fan $\F_{Q_{\hat{P}}}$ of the root polytope $\NP(Q_{\hat{P}})$
of the starred quiver $Q_{\hat{P}}$
coincide with the rays of the (inner) normal fan $\NF(\OO(P))$ of 
the order polytope $\OO(P)$;  these rays are in bijection 
with arrows of $Q_{\hat{P}}$, or equivalently, with cover relations in 
the Hasse diagram of $\hat{P}$.
\end{lemma}

\begin{proof} 
Let $P=\{\vv_1,\dots,\vv_n\}$.
It follows from 
\Cref{def:posetquiver},
\Cref{def:NP}, and \Cref{lem:vertexarrow}
that the vertices of $\NP(Q_{\hat{P}})$, and hence
	the rays of $\F_{Q_{\hat{P}}}$,
are in bijection with the cover relations in 
$\hat{P}$.  Specifically,
\begin{itemize}
\item For each cover relation $\vv_i \lessdot \vv_j$,
we have the vertex $e_j-e_i$ of $\NP(Q_{\hat{P}})$.
\item For each cover relation $\hat{0} \lessdot \vv_j$,
we have the vertex $e_j$.
\item For each cover relation $\vv_i \lessdot \hat{1}$,
we have the vertex  $-e_i$.
\end{itemize}
Each vertex $u$ of $\NP(Q_{\hat{P}})$ listed above 
 gives rise to the 
ray $\overrightarrow{0u}$
	from $0$ to $e_j-e_i$ (or $e_j$ or $-e_i$) in the face fan $\F_{Q_{\hat{P}}}$.

Meanwhile, the facet inequalities of $\OO(P)$
are also in bijection with the cover relations in 
$\hat{P}$.  If we identify $\R^{P}$ with 
$\R^n$ in the natural way, we get the following facet inequalities.
\begin{itemize}
\item For each cover relation $\vv_i \lessdot \vv_j$,
we have the facet inequality $x_j-x_i \geq 0$.
The linear functional $w$ which takes the dot product with 
$(e_j-e_i)$ is minimized  precisely at this facet.
\item For each cover relation $\hat{0} \lessdot \vv_j$,
we have the facet inequality $x_j \geq 0$.
The linear functional $w$ which takes the dot product with 
$e_j$ is minimized precisely at this facet.
\item For each cover relation $\vv_i \lessdot \hat{1}$,
we have the facet inequality $1-x_i \geq 0$.
The linear functional $w$ which takes the dot product with 
$-e_i$ is minimized precisely at this facet.
\end{itemize}

Therefore we see that we have an identification of rays
	of the face fan $\F_{Q_{\hat{P}}}$ with rays of the 
 (inner) normal fan $\NF(\OO(P))$.
\end{proof}

\begin{theorem}\label{thm:refine}
Let $P=\{\vv_1,\dots,\vv_n\}$ be a finite ranked poset.
	Then the face fan $\F_{Q_{\hat{P}}}$ of the root polytope $\NP(Q_{\hat{P}})$
 of the starred quiver $Q_{\hat{P}}$
refines the (inner) normal fan $\NF(\OO(P))$ of 
the \emph{order polytope} $\mathcal{O}(P)$ of $P$:
the two fans have the same set of rays, and each 
maximal cone of $\NF(\OO(P))$ is a union of maximal cones of 
	$\F_{Q_{\hat{P}}}$.  Moreover, if $P$ is a graded poset, then 
	$\F_{Q_{\hat{P}}}$ coincides with 
 $\NF(\OO(P))$.
\end{theorem}

We note that if $P$ is not ranked, then the refinement statement of \Cref{thm:refine} may fail. Namely, the right-hand side image in  
\Cref{fig:nonranked} shows an example where it fails.

\begin{figure}
\[\begin{tikzcd}
	&&&&&& {\star_1} \\
	&& {\star_1} && {\overset{v'_7}\bullet} && {\overset{v'_8}\bullet} \\
	{\overset{v_6}\bullet} && {\overset{v_7}\bullet} &&& {\overset{v'_6}\bullet} \\
	& {\overset{v_5}\bullet} &&&&& {\overset{v'_5}\bullet} \\
	&& {\overset{v_4}\bullet} &&&&& {\overset{v'_4}\bullet} \\
	&&& {\overset{v_3}\bullet} &&&&& {\overset{v'_3}\bullet} \\
	&& {\overset{v_1}\bullet} && {\overset{v_2}\bullet} && {\overset{v'_1}\bullet} &&& {\overset{v'_2}\bullet} \\
	&& {\star_0} &&&& {\star_0}
	\arrow["2", curve={height=-6pt}, from=2-5, to=1-7]
	\arrow["{-1}", Rightarrow, thick, from=2-7, to=1-7]
	\arrow["2", curve={height=-6pt}, from=3-1, to=2-3]
	\arrow["{-1}", Rightarrow, thick, from=3-3, to=2-3]
	\arrow["{-1}"', Rightarrow, thick, from=3-6, to=2-5]
	\arrow["{-1}"', Rightarrow, thick, from=4-2, to=3-1]
	\arrow["{-1}"', Rightarrow, thick, from=4-7, to=3-6]
	\arrow["{-1}"', Rightarrow, thick, from=5-3, to=4-2]
	\arrow["0"', from=5-8, to=4-7]
	\arrow["{-1}", Rightarrow, thick, from=6-4, to=5-3]
	\arrow["{-1}", Rightarrow, thick, from=6-9, to=5-8]
	\arrow["{-1}", curve={height=-18pt}, Rightarrow, thick, from=7-3, to=3-1]
	\arrow["{-1}", curve={height=18pt}, Rightarrow, thick, from=7-5, to=3-3]
	\arrow["{-1}", Rightarrow, thick, from=7-5, to=6-4]
	\arrow["{-1}"', curve={height=-18pt}, Rightarrow, thick, from=7-7, to=2-5]
	\arrow["{-1}", curve={height=18pt}, Rightarrow, thick, from=7-10, to=2-7]
	\arrow["{-1}", Rightarrow, thick, from=7-10, to=6-9]
	\arrow["{-1}", Rightarrow, thick, from=8-3, to=7-3]
	\arrow["2"', curve={height=12pt}, from=8-3, to=7-5]
	\arrow["{-1}", Rightarrow, thick, from=8-7, to=7-7]
	\arrow["2"', curve={height=12pt}, from=8-7, to=7-10]
\end{tikzcd}\]
\caption{The above quivers $Q_{\hat {P}}$ and $Q_{\hat P'}$ are constructed out of two posets, $P$ and $P'$, neither of which is ranked. At left: a facet arrow-labeling of $Q_{\hat {P}}$ for which the facet components $C_0$ of $\star_0$ and $C_1$ of $\star_1$ are not distinct. At right: a facet arrow-labeling of $Q_{\hat {P}'}$ whose facet component 
	$C_0$ is not a filter; hence the corresponding maximal cone of $\F_{Q_{\hat{P}'}}$ does not lie in a maximal
	cone of $\NF(\OO(P'))$. In both examples the double arrows show the arrows labeled~$-1$}
	\label{fig:nonranked}
\end{figure}
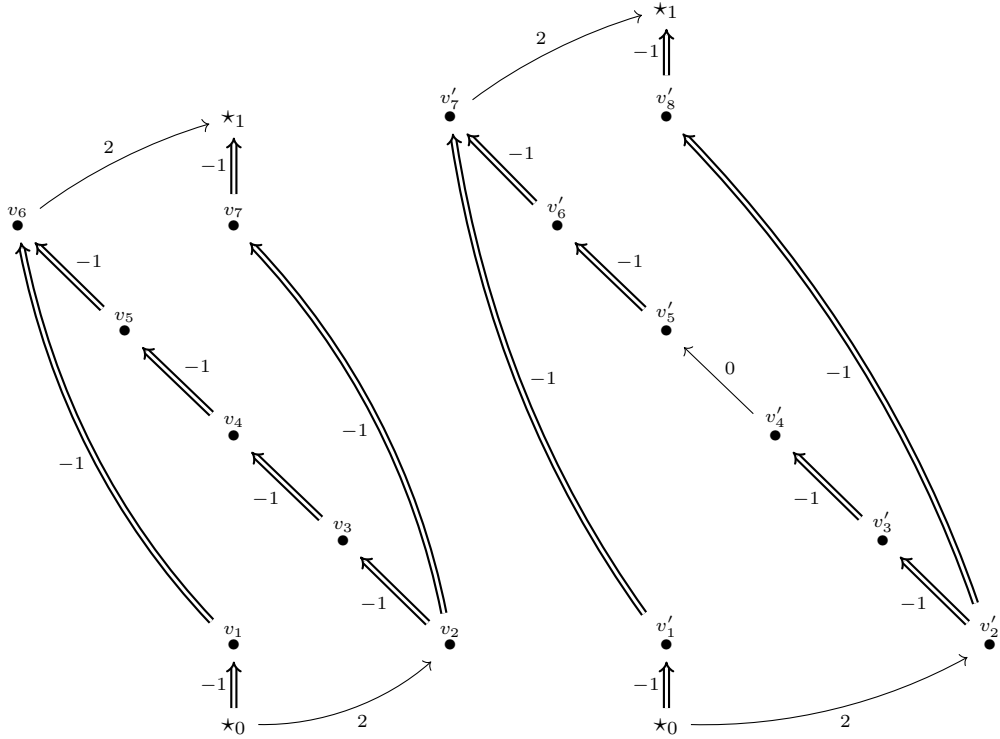

Before proving the theorem, we first 
need the following lemma. Recall the  facet arrow-labelings and corresponding $\bullet$-labelings from Section~\ref{sec:root} that are used for describing the facets of root polytopes. Recall also the facet components associated to facet arrow-labelings in \cref{def:component}

\begin{lemma} \label{lem:orderideal}
Suppose that $P=\{\vv_1,\dots,\vv_n\}$ is a ranked poset. 
Let $M$ be a facet arrow-labeling of the starred quiver $Q_{\hat{P}}$, and $L$ the corresponding $\bullet$-labeling. Consider the associated facet components  $C_0$ and $C_1$ of $Q_{\hat{P}}$ containing $\star_0$ and $\star_1$, respectively.  

The facet components $C_0$ and $C_1$ are disjoint. The facet component $C_0$ consists of a set of vertices $S\sqcup\{\star_0\}$ along with all arrows between them, where the set $S$ forms a lower order ideal in $P$. For the facet component $C_1$, the set of normal vertices in $C_1$ 
	forms a filter in $P$,
	namely $\{\vv_1,\dots,\vv_n\} \setminus S$. 

\end{lemma}
\begin{remark} 
Note that as the lemma says, the vertices of $C_0$ form a lower order ideal in $\hat{P}$
and the vertices of $C_1$ form a filter in $\hat{P}$.
The lemma also says that $C_0$ is a full subquiver of $Q_{\hat{P}}$, though, while $C_1$ need not be.
	This reflects the asymmetry inherent in the definition of a ranked poset. Note also that $C_0$ and $C_1$ need not be distinct if $P$ is not ranked. An example of this is shown in Figure~\ref{fig:nonranked}, on the left-hand side.
\end{remark}
\begin{proof}
Note that we can compute the vertex coordinates of all vertices of 
$Q_{\hat{P}}$ that lie in $C_0$, since the vertex $\star_0$ (which gets
vertex-coordinate equal to $0$) lies in $C_0$ by definition and 
whenever two vertices are connected by an arrow in $C_0$, their 
vertex coordinates differ by $-1$.

Since $P$ is ranked, we can use this to show for each vertex $\vv_j$ in $C_0$ that
$L(\vv_j) = -\rank(\vv_j)$. Namely, let $r=\rank(\vv_j)$ and 
choose a path in $C_0$ (not necessarily oriented) connecting
$\star_0$ to $\vv_j$.
Going from $\star_0$ to $\vv_j$ along this path, if we traverse a total of $s$
edges in the forward direction, then $s\ge r$ and we must also traverse a  total of
$(s-r)$ edges in the backward direction (to be able to reach the rank $r$ element $\vv_j$).  Computing vertex coordinates along the way, we get 
$L(\vv_j) = -s+(s-r)=-\rank(\vv_j)$.
Note that this also implies that $\star_1$ cannot lie in $C_0$,
	since $L(\star_1)=0$.

We can now deduce the description of $C_0$. Consider an element $\vv_j$ of $S$ and suppose we have another element $\vv_\ell\in P$ with $\vv_\ell<\vv_j$.  Note that $k:=\rank(\vv_\ell)<\rank(\vv_j)=r$. Choose a 
	maximal chain in $\hat{P}$ 
between $\hat 0$ and $\vv_\ell$. This maximal chain has $k+1$ elements (including $\hat 0$ and $\vv_\ell$), and in the quiver $Q_{\hat{P}}$ it corresponds to a directed path from $\star_0$ to $\vv_{\ell}$. We furthermore consider a maximal chain between $\vv_\ell$ and $\vv_j$ and its associated directed path in $Q_{\hat{P}}$. The concatenation of the two paths gives a directed path in $Q_{\hat{P}}$ that goes from $\star_0$ to $\vv_j$, passes through $\vv_\ell$, and has overall length $r$.    
For each arrow $a$ of this path we have an arrow label $M(a)\ge -1$, and the sum of the arrow labels equals to $L(\vv_j)$. Moreover, since $\vv_j\in S$ we know  that $L(\vv_j) =-\rank(\vv_j)=-r$. Since our path is made up of precisely $r$ arrows, to get the sum $-r$ we must have $M(a)=-1$ for each arrow $a$. Therefore the entire path lies in the facet component $C_0$. In  particular, it follows that $\vv_\ell$ lies in $C_0$. Since $\vv_\ell$ was in $P$ we now have $\vv_\ell\in S$. We see that $S$ is an order ideal for $P$. 
 
Now suppose $\vv_i$ and $\vv_j$ in $S$ are connected by an arrow $a$, so $\vv_i\lessdot v_j$. Applying the argument above with $\vv_\ell=\vv_i$ we obtain a directed path in $Q_{\hat{P}}$ going from $\star_0$ to $\vv_j$  and passing through $\vv_i$, that lies entirely in $C_0$. This path necessarily contains the arrow $a$ and therefore  $a$  lies in $C_0$ as claimed.

Finally, let us consider the facet component $C_1$ of $\star_1$. 
We saw earlier that $\star_1$ cannot be in $C_0$, so
	$C_0$ and $C_1$ are  distinct facet components. 
Moreover, by \Cref{lem:containstar}, each facet component of $Q_{\hat{P}}$ contains some starred vertex, so each vertex of $Q_{\hat{P}}$ must lie either in $C_1$ or in $C_0$. 
Since the vertices of $S$ form a lower order ideal in $P$, the complementary set of vertices (namely the normal vertices of $C_1)$ must form a filter in $P$.
\end{proof}

\begin{remark}\label{rem:graded}
If the poset $P$ in \Cref{lem:orderideal} is not just ranked but also graded,
then the proof in \Cref{lem:orderideal} extends to show that 
the facet component $C_1$ consists of 
the vertices $\{\star_1\} \cup 
\{\vv_1,\dots,\vv_n\} \setminus S $ together with 
all arrows 
joining two elements in this set.
\end{remark}

\begin{proof}[Proof of \Cref{thm:refine}]
The statement that the rays of the two fans coincide was already proved in \cref{lem:rays}. To show that each maximal cone of $\NF(\OO(P))$ is a union of maximal cones of 
	$\F_{Q_{\hat{P}}}$, it suffices to show that each maximal
	cone of $\F_{Q_{\hat{P}}}$ is contained
in a maximal cone of $\NF(\OO(P))$. Equivalently, we need to show 
	that the rays of each maximal cone of $\F_{Q_{\hat{P}}}$ 
 are a subset of the rays
of some 
 maximal cone of $\NF(\OO(P))$.

	Each maximal cone $c_M$ of $\F_{Q_{\hat{P}}}$ comes from a facet arrow-labeling $M$
of $Q_{\hat{P}}$, and the rays in $c_M$ are indexed
by those 
 arrows $a\in \Arr(Q_{\hat{P}})$ for which 
$M(a) = -1$. These are the arrows appearing in 
the facet components
of $Q_{\hat{P}}$ (with respect to $M$).

By \Cref{lem:containstar}, each facet component of $Q_{\hat{P}}$
contains some starred vertex, so the vertices of $Q_{\hat{P}}$ 
lie in either $C_1$, the facet component of $\star_1$,
or $C_0$, the facet component of $\star_0$.
By \Cref{lem:orderideal}, the vertices of the 
facet component $C_0$ form a lower order ideal of $P$, and the 
vertices of $C_1$ form a filter of $P$.
Therefore each arrow indexing a ray in $c_M$ -- 
which is by definition an arrow appearing in a facet components of $Q_{\hat{P}}$ --
connects either two vertices in the filter $I$, or two vertices
in the complement of the filter $\{\vv_1,\dots,\vv_n\} \setminus I$.
By \Cref{rem:graded},  if $P$ is graded,
then the arrows appearing in $C_i$ (for $i=0$ or $1$)
comprise \emph{all} arrows connecting two vertices in $C_i$.

Meanwhile, each maximal cone $c'_I$ of $\NF(\OO(P))$ 
corresponds to a 
vertex 
 of $\OO(P)$, which is the 
characteristic function $\chi_I$ of a filter $I$ of $P$.
The rays in $c'_I$ correspond to the 
edges $\vv_i \lessdot \vv_j$ of the Hasse diagram of $\hat{P}$,
where $\chi_I(\vv_i) = \chi_I(\vv_j)$.
Therefore each ray of $c'_I$ corresponds to an arrow in $\Arr(Q_{\hat{P}})$
which connects either two vertices in the filter $I$
or two vertices in the complement of the 
filter $\{\vv_1,\dots,\vv_n\} \setminus I$.

This shows that the set of rays in a maximal cone 
$c_M$ 
	of $\F_{Q_{\hat{P}}}$ 
 are a subset of the rays
of a corresponding maximal cone 
 $c'_I$ of $\NF(\OO(P))$, with equality when $P$ is graded.
\end{proof}

If we combine \Cref{thm:refine} and 
\Cref{t:reflexive}, we obtain the 
following result, which also follows from 
\cite[Remark 1.6]{HH}.
\begin{corollary}
Let $P$ be a graded finite poset.  Then 
$\NF(\OO(P))$ coincides with 
	$\F_{Q_{\hat{P}}}$.
It follows that the order polytope $\OO(P)$ is combinatorially equivalent
to the polar dual of the root polytope $Q_{\hat{P}}$.
\end{corollary}

\section{Applications in mirror symmetry and toric geometry} 
\label{sec:toric}

In this final section 
we start by explaining how our previous results 
 are related to mirror symmetry and toric geometry.  
We show that when $Q$ is a strongly-connected
starred quiver, the toric variety associated to the fan $\F_Q$ has a {small} toric desingularisation. 
 We also apply our constructions
 to describe the Picard group for toric varieties arising from quivers, with a particular focus
 on the quivers coming from ranked posets.

\subsection{Quiver Laurent polynomials, Newton polytopes,
and superpotential polytopes}
In the study of mirror symmetry for Fano varieties such as partial flag varieties, quiver flag varieties, and Grassmannians,
there are naturally associated Laurent polynomial \emph{superpotentials}
which have the form 
\begin{equation}\label{eq:super}
S(x_1,\dots,x_n) \in 
	\CC[x_1^{\pm 1},\dots,x_n^{\pm 1}][q_1,\dots,q_d].
\end{equation}
Given such a superpotential, there are two polytopes that one can associate,
the \emph{Newton polytope} and the \emph{superpotential polytope}.
\begin{definition}[Newton polytope]
The \emph{Newton polytope} $\Newt_S \subset \R^n$
is the convex hull of the exponent vectors of the Laurent monomials
in $S(x_1,\dots,x_n)$, where the exponent vector
	of $c(q_1,\dots,q_d) x_1^{a_1} \dots x_n^{a_n}$ for $c(q_1,\dots,q_d)
	\in \CC[q_1,\dots,q_d]$ is 
$(a_1,\dots,a_n)$.
\end{definition}

The above definition defines a polytope as the convex hull of points,
one for each Laurent monomial in $S(x_1,\dots,x_n)$.  On the other
hand, we can define a different polytope -- the \emph{superpotential polytope} -- by \emph{tropicalizing} the superpotential.  This cuts out a 
polytope that lies in the dual space by inequalities associated to the Laurent monomials in $S(x_1,\dots,x_n)$. Our terminology is following \cite{RW}. 

\begin{definition}[Superpotential polytope]=
\label{def:super}
Let 
$S(x_1,\dots,x_n) \in 
\CC[x_1^{\pm 1},\dots,x_n^{\pm 1}][q_1,\dots,q_d]$ be a Laurent polynomial
with positive coefficients,
and choose
real numbers $\mathbf{r} = (r_1,\dots,r_d)$.
Set $\Trop(x_i)=X_i$ and $\Trop(q_i)=r_i$.
We now inductively define the tropicalization 
$\Trop(S(x_1,\dots,x_n))$
of $S(x_1,\dots,x_n)$
by requiring that if 
$\mathbf{h_1}, \mathbf{h_2}\in 
\CC[x_1^{\pm 1},\dots,x_n^{\pm 1}][q_1,\dots,q_d]$ are Laurent polynomials
	with positive 
	coefficients and $a_1,a_2\in \R_{>0}$, then 
	$$\Trop(a_1 \mathbf{h_1} + a_2 \mathbf{h_2}) = 
	\min(\Trop(\mathbf{h_1}),\Trop(\mathbf{h_2})) \ \text{ and } \ 
	\Trop(\mathbf{h_1 h_2}) = \Trop(\mathbf{h_1}) + \Trop(\mathbf{h_2}).$$
We define the \emph{superpotential polytope} $\Gamma_S^{\mathbf{r}} = \Gamma_S \subset \R^n$
by  $\Trop(S(x_1,\dots,x_n)) \geq 0.$  In other words,
we impose one inequality
$$(\sum_{i=1}^d {\ell_i}r_i) +  (\sum_{j=1}^{n} m_j X_j) \geq 0$$
for each Laurent monomial summand 
$\prod_{i=1}^d q_i^{\ell_i} \prod_{j=1}^{n} x_j^{m_j}$
of $S(x_1,\dots,x_n)$.
\end{definition}

\begin{example}\label{ex:twopolytopes}
Let $$S(x_1,\dots,x_6)=x_1 +\frac{x_2}{x_1}+\frac{x_3}{x_2}+\frac{x_4}{x_3}+\frac{q_1}{x_1}+\frac{x_5}{x_1}+\frac{x_6}{x_2}+\frac{x_6}{x_5}+\frac{q_2}{x_6}.$$
Then the Newton polytope
is $$\Newt_S = \conv(e_1, e_2-e_1, e_3-e_2,
e_4-e_3, -e_1, e_5-e_1, e_6-e_2, e_6-e_5, -e_6) \subset \R^6.$$
        Meanwhile, if we let $\mathbf{r}=(r_1,r_2)=(1,1)$, then the
        superpotential polytope  $\Gamma_S^{\mathbf{r}} \subset \R^6$ is cut
        out by the inequalities
\begin{align*}
        &X_1  \geq 0 &\qquad   &X_2 - X_1  \geq 0  &\qquad  & X_3-X_2  \geq 0 \\
         &X_4 -X_3  \geq 0 &\qquad   & 1-X_1  \geq 0 &\qquad &X_5 -X_1  \geq 0\\
         &X_6-X_2  \geq 0 &\qquad &X_6-X_5  \geq 0  &\qquad &1-X_6  \geq 0
\end{align*}
\end{example}

In the context of mirror symmetry for the Fano varieties mentioned at the start of this section,
many of the associated superpotentials can be read off from a 
strongly-connected starred quiver $Q$
(cf \cref{def:starconnected}). 
In particular, such a quiver  $Q$ 
gives rise
to the Laurent polynomial 
\begin{equation}\label{e:SQ}S_Q(x_1,\dots,x_n):= \sum_{a:\vv_i\to \vv_j} 
\frac{x_j}{x_i}
+\sum_{a:\star_i\to \vv_j} 
\frac{1}{\wt(\star_i)} \cdot x_j
+ \sum_{a: \vv_i\to \star_j} 
\wt(\star_j) \cdot \frac{1}{x_i},
\end{equation} 
where we sometimes refer to $\wt(\star_i)$ as a `quantum parameter'
and write $\wt(\star_i) = q_i$.
Note that the above expression consists of 
 one `head over tail' Laurent monomial for each arrow of $Q$,
with a coordinate $x_i$ associated to each $\bullet$-vertex $\vv_i$.
As an example, the superpotential from \cref{ex:twopolytopes} comes from
the quiver in \cref{fig:Schubert}.
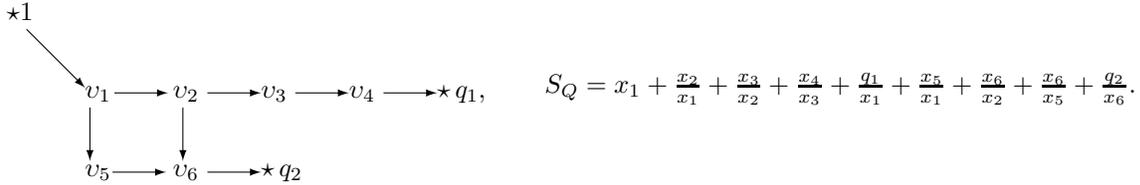
\begin{figure}[h!]
\begin{center}
\[\adjustbox{scale=.80}
{\begin{tikzcd}
	{\star 1} \\
	& {\vv_1} & {\vv_2} & {\vv_3} & {\vv_4} & {\star q_1} \\
	& {\vv_5} & {\vv_6} & {\star q_2}
	\arrow[from=1-1, to=2-2]
	\arrow[from=2-2, to=2-3]
	\arrow[from=2-2, to=3-2]
	\arrow[from=2-3, to=2-4]
	\arrow[from=2-3, to=3-3]
	\arrow[from=2-4, to=2-5]
	\arrow[from=2-5, to=2-6]
	\arrow[from=3-2, to=3-3]
	\arrow[from=3-3, to=3-4]
\end{tikzcd}}
\quad\ S_Q=x_1 +\frac{x_2}{x_1}+\frac{x_3}{x_2}+\frac{x_4}{x_3}+\frac{q_1}{x_1}+\frac{x_5}{x_1}+\frac{x_6}{x_2}+\frac{x_6}{x_5}+\frac{q_2}{x_6}.\]
 \end{center}
	\caption{The starred quiver $Q$ and the Laurent polynomial $S_Q$.
	This Laurent polynomial is
	associated to the Schubert variety $X_{(4,2)}$
	in the Grassmannian $Gr_{2}(\C^8)$ \cite{RW2}\label{fig:Schubert}}
\end{figure}

\subsection{The toric variety $Y(\F_Q)$ and connections with mirror symmetry} 
Let us consider the Laurent polynomial  $S_Q(x_1,\dots,x_n)$ coming from a starred quiver $Q$. Note that the Newton polytope 
$\Newt_{S_Q}$ of the quiver Laurent polynomial 
$S_Q(x_1,\dots,x_n)$ is precisely 
the root polytope $\NP(Q)$ described in 
\cref{def:NP}. Therefore when $Q$ is 
 strongly-connected, we know from
\cref{thm:reflexive} that 
$\Newt_{S_Q} = \NP(Q)$ is reflexive and terminal. Letting $\F_Q$ denote the face fan of $\NP(Q)$, we can now associate the Gorenstein Fano toric variety $Y(\F_Q)$ to $S_Q$, and we call $S_Q$ the mirror superpotential for $Y(\F_Q)$. Note that $Y(\F_Q)$ is not quite smooth in general, since it may have terminal singularities, and not all aspects of mirror symmetry make sense directly for $Y(\F_Q)$. 

In the smooth case, study of the Laurent polynomial mirrors for toric Fano manifolds $Y$ and their properties (relations to $I$-functions, quantum cohomology) goes back to \cite{Batyrev0, Batyrev, Givental:ICM, Givental:toricCI, Givental:equivariant, OT, HoriVafa2000}. More generally, if  $X$ is a smooth Fano variety with a suitable flat toric degeneration whose central fiber is $X_0=Y(\F_Q)$, then $S_Q$ may also encapsulate Gromov-Witten invariants (`quantum periods') of~$X$, see \cite[Section~4]{KP:whyandhow}. See also \cref{ex:Galkin} and \cite[Example~5.7]{Galkin}, and more generally \cite[Conjecture 3.3 and Remark 3.4]{Golyshev}. We note also that $S_Q$ is automatically `maximally mutable', see \cite[Section~5]{CKPT}.

An early example of quiver Laurent polynomial mirrors  was Givental's quiver Laurent polynomial superpotential for the full flag variety \cite{Givental:fullflag} and \cite{BC-FKvSGrass, BC-FKvS} which gave partial flag variety analogues and related the construction to toric degenerations. These examples have also been extended to some homogeneous spaces in other Lie types, such as quadrics and maximal orthogonal Grassmannians \cite{PRW, PR, Spacek, SW1, SW2}. Also Kalashnikov~\cite{Kalashnikov} constructed Laurent polynomial mirrors for certain quiver flag varieties using toric degenerations. 

Most recently we constructed mirrors for Grassmannian Schubert varieties that restrict to quiver-Laurent polynomials on an appropriate chart \cite{RW2}; \cref{fig:Schubert} shows an example. Though Grassmannian Schubert varieties are not smooth, for the calculation of quantum periods of smooth Calabi-Yau subvarieties of a variety $X$, the smoothness assumption on $X$ can be relaxed. See \cite{Miura:minuscule}, which used Laurent polynomial mirrors of the form $S_Q$ to study mirror symmetry for Calabi-Yau $3$-folds in Gorenstein Fano Schubert varieties.

On the symplectic side, if there is a symplectic Fano manifold $X$ with a degeneration to $Y(\F_Q)$ satisfying certain technical conditions then the Laurent polynomial $S_Q$ obtains a Floer theoretic interpretation, see \cite[Theorem~1]{Nishinou}, and \cite[Theorem 4.4]{BGM22}. Note that \cite[Theorem~1]{Nishinou} requires the existence of a small resolution of the central toric fiber. 
Our first result in this section will be to prove the existence of such a resolution  for $Y(\F_Q)$, 
see \cref{p:desing}.
In this symplectic context, the existence of a degeneration to a $Y(\F_Q)$ (and the resulting Floer-theoretic interpretation of $S_Q$) then implies the existence of a non-displaceable Lagrangian torus in $X$, see \cite[Corollary~2]{Nishinou}.

\subsection{A small resolution of the toric variety $Y(\F_Q)$}\label{s:smalldesing}

In this section we show that 
when $Q$ is a strongly-connected
starred quiver, 
we can resolve the singularities of $Y(\F_Q)$ to get a smooth toric variety without changing the rays of the fan. 
In other words, there is a \textit{small} toric desingularisation of $Y(\F_Q)$. Such a resolution was constructed explicitly for a specific example in \cite[Section 3]{BC-FKvS}. We deduce its existence in our more general setting by making use of the \emph{maximal projective crepant partial (MPCP)} 
desingularisation of \cite{Batyrev}, which relies on a construction from \cite{GKZ89}. 

\begin{thm}\label{p:desing} 
	Consider any strongly-connected starred quiver $Q$. Let $\F_Q$ be the  face fan of the root polytope $\NP(Q)$. There exists a refinement $\widehat \F_Q$ of $\F_Q$ such that  $Y(\widehat \F_Q)\to Y(\F_Q)$ is a small crepant toric desingularisation.  
\end{thm}

\begin{proof} By 
\cref{cor:terminalGorenstein}
	we know that $Y(\F_Q)$ is Gorenstein with at most terminal singularities.  By
 \cite[Theorem 2.2.24]{Batyrev}, any Gorenstein toric variety has a MPCP desingularization, which will be terminal and $\mathbb Q$-factorial. Since $Y(\F_Q)$ was already terminal, it follows that any crepant partial resolution of $Y(\F_Q)$ is necessarily small. Indeed, there is no way to add new rays and stay crepant if the faces of the fan polytope have no interior lattice points. 
 Now we have a new fan $\widehat\F_Q$ which is simplicial, refines $\F_Q$, and has the same rays as $\F_Q$. Thus the primitive ray generators are vertices $u_a$ of the root polytope $\NP(Q)$. Consider now a maximal cone $C$ of $\widehat\F_Q$, and let $\{u_a\mid a\in \Arr_C\}$ denote the set of primitive generators of the rays belonging to $C$. The convex hull of $\{u_a\mid a\in \Arr_C\}\cup\{0\}$ is a simplex. Consider the quiver $Q'$ obtained from $Q$ by removing the arrows not belonging to $\Arr_C$. The associated matrix whose row vectors are the vertices of $\NP(Q')$, is square (since the cone $C$ was simplicial) and totally unimodular, see \cref{r:TU}. Therefore its determinant must be $\pm 1$ or $0$. However, it cannot be $0$ since the cone was full-dimensional. Therefore, it must be $\pm 1$ which implies that the cone $C$ is regular. It follows that $Y(\widehat\F_Q)$ is smooth.  
 \end{proof}

Recall that a lattice polytope $\mathbf P\in\R^n$ is said to have the \emph{integer decomposition property (IDP)} if for any $k\in\Z_{>0}$ and any element $\vv\in (k\mathbf P)\cap\Z^n$, the element $\vv$ can be represented as the sum of $k$ lattice points from $\mathbf P$. 

\begin{cor}\label{c:RootIDP} For any strongly-connected starred quiver $Q$, the polytope $\NP(Q)$ has a triangulation into unimodular simplices and satisfies the integer decomposition property. 
\end{cor}

\begin{proof} A triangulation of $\NP(Q)$ into unimodular simplices arises in the proof of \cref{p:desing}. Namely, for any cone of $\widehat\F_Q$ the convex hull of the primitive generators of the rays together with $\mathbf 0$ gives a unimodular simplex. These form the simplices of the triangulation. Each of these has the IDP, and the second statement also follows.  
\end{proof}

\begin{remark}\label{r:rootIDP} In fact, it is possible to show in another way that the polytope $\NP(Q)$ has a unimodular triangulation (and thus the IDP), whenever $Q$ contains an oriented cycle or an oriented path $\pi$ from a starred vertex to a starred vertex. Namely, in this case the origin $\mathbf 0$ lies in $\NP(Q)$ (since $\sum_{a\in\pi} u_a=\mathbf 0$), and whenever  $\NP(Q)$ contains the origin, it has a unimodular triangulation by \cite[Theorem 3.9]{HPPS}. 
\end{remark}

\begin{remark}\label{rem:superIDP}
All superpotential polytopes (see \cref{def:super})
have the integer decomposition property. 
The normal vectors to the facets of a superpotential polytope are the vertices
of the root polytope, and the matrix encoding the vertices of the root polytope is 
totally unimodular, hence we can apply \cite[Theorem~2.4]{HPPS}. 
This says that each superpotential polytope has a regular unimodular triangulation.
\end{remark}

\subsection{The toric variety $Y(\F_Q)$ when $Q$ is planar or comes
from a ranked poset}
The relationship between quiver Laurent polynomials $S_Q$ and toric geometry is particularly beautiful when the quiver $Q$ is planar, or when
it comes from a ranked poset, as we will explain; the quivers we study in \cite{RW2} are both.

\subsubsection{When $Q$ is planar}
When the quiver $Q$ is planar (as it is in \cite{RW2}), we get an interpretation
of the toric variety $Y(\F_Q)$ as a \emph{toric quiver variety} for the planar dual quiver $Q^\vee$, i.e. a 
\emph{quiver moduli space}
\cite{King:quivermoduli} for $Q^{\vee}$ parameterizing representations of 
$Q^{\vee}$ with dimension vector $(1,1,\dots,1)$.
To see this, recall that the projective variety associated to the flow polytope $\Fl_{Q^{\vee}}$
for $Q^{\vee}$ is the quiver moduli space (for the canonical weight) associated to $Q^{\vee}$
 \cite{AltmannvanStraten}.  This is also the toric variety associated to the normal fan 
 $\NF(\Fl_{Q^{\vee}})$.
Since $\F_Q$ agrees with $\NF(\Fl_{Q^{\vee}})$
(by \cref{mainthm:dual}), we obtain an interpretation of 
the toric variety $Y(\F_Q)$ as a quiver moduli space for $Q^{\vee}$.
 These varieties have been studied extensively, see  \cite{Hille:toricquivervarieties,
 Hille:qcp,
 Hille,
 Joo:thesis,
 ANSW,DomokosJoo,nasr2022smooth} as well as in \cite{Kalashnikov,CDK} and \cite{Miura:CYinHibi, Miura:thesis} where plane quivers and planar duality play a role. 

\subsubsection{When $Q$ comes from a ranked poset}\label{s:hibiintro}
Suppose that $Q$ comes from the bounded extension $\hat{P}$ of a poset $P$ 
(as in \cref{rem:quiverfromposet}), and we let the weight of the starred vertices
associated to $\hat{1}$ and $\hat{0}$ be $q$ and $1$, respectively.
Then  the {superpotential polytope}
$\Gamma_Q^{\mathbf{r}}:=\Gamma_{S_Q}^{\mathbf{r}}$  
with $\mathbf{r}=(1)$, let us denote it $\Gamma_Q$,
agrees with the order polytope $\OO(P)$ of $P$.
For example, the quiver from \cref{fig:Schubert} can be identified with
the Hasse diagram of 
a poset $P$ on $\{\vv_1,\dots,\vv_6\} \cup \{\hat{0},\hat{1}\}$,
where the starred vertex labeled $1$ is identified with $\hat{0}$,
the starred vertices labeled  $q_1, q_2$ are identified with $\hat{1}$,
and the arrows point from smaller to larger elements in the poset, and the superpotential polytope $\Gamma_Q$ is exactly the order polytope of $P$, see 
\cref{ex:twopolytopes}.
The fact that the superpotential polytope agrees with the order polytope 
 allows one to deduce 
nice properties of $\Gamma_Q$, e.g. its volume is the number of linear
extensions of the poset \cite{order}.

Given that there are two polytopes that we can  associate
to the quiver Laurent polynomial $S_Q$ -- namely, the 
 Newton polytope $\Newt_{S_Q} = \NP(Q)$ of $S_Q$ and the
superpotential polytope $\Gamma_Q$
-- it is 
natural to ask how these two polytopes are related.  
The answer is provided by 
our \cref{thm:D} (see also \cref{mainthm:poset}):
when $Q$ comes from a ranked poset $P$,
the face fan 
$\mathcal{F}_Q$ of $\NP(Q)$
refines the (inner) normal fan $\NF(\Gamma_Q) = \NF(\OO(P))$ while preserving the rays,
and when $P$ is graded, we get equality.
It follows that when $P$ is ranked, 
the toric variety $Y(\mathcal{F}_Q)$ of $\mathcal{F}_Q$ provides a
small partial desingularization
of the toric variety $Y(\NF(\Gamma_Q))$ associated to the normal fan of the superpotential polytope, see \cref{fig:motivation}  
for a summary. Note that $Y(\NF(\Gamma_Q))=Y(\NF(\OO(P)))$ is the projective toric variety associated to $\OO(P)$, also denoted $Y_{\OO(P)}$, and is commonly known as a \emph{Hibi toric variety} as it agrees with the toric variety introduced in \cite{hibi1987distributive}.

\begin{figure}[h]
\[\begin{tikzpicture}[node distance=2cm, auto]
  \node[draw, align=center] (C) {$Q=Q_{\hat P}$  from ranked poset $P$\\ with quiver Laurent polynomial $S_Q$};
  \node[draw, align=center, xshift=-1cm] (P) [below left of=C, yshift=-.3cm] {$\NP(Q)=$ Newton\\ polytope of $S_Q$};
  \node[draw,align=center,xshift=1.5cm,yshift=-.3cm] (Ai) [below right of=C] {Order polytope $\mathcal O(P)=$\\superpotential polytope of $S_Q$};
  \draw[->] (C) to node {} (Ai);
  \draw[->] (C) to node [swap] {} (P);
\node[draw,yshift=.5cm, align=center] (T) [below of=P] 
{Face fan\\$\F(\NP(Q))$};
\node[xshift=.5cm] (S) [right of=T] {refines};
\node[draw,xshift=-.5cm,yshift=.5cm,align=center] (U) [below of=Ai] {Normal fan\\$\NF(\mathcal O(P))$};
\end{tikzpicture}
\]
	\caption{How the root polytope and order polytope are related to 
	the Laurent polynomial superpotential $S_Q$ of the quiver associated
	to a ranked poset}
	\label{fig:motivation}
\end{figure}
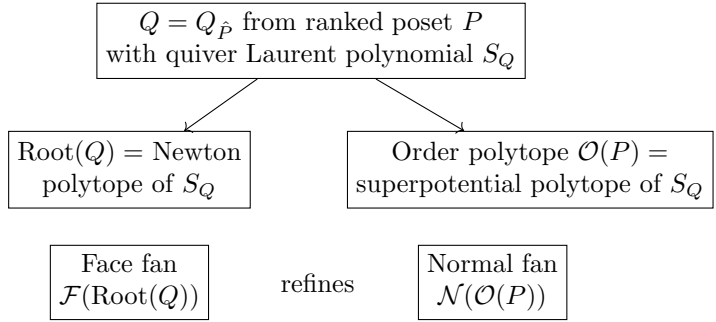

We now show how to describe the Picard group of the toric variety $Y(\F_Q)$, especially in the case where the 
starred quiver $Q$ comes from a ranked poset.

\subsection{The Picard group
$\Pic(Y(\mathcal F_Q))$ and a canonical extension of a ranked poset}\label{sec:RankedCartier}

Recall that $Y(\mathcal F_Q)$ denotes the toric variety associated to  $\F_Q$, 
the face fan of $\NP(Q)$. We now consider the case where $Q$ is the starred quiver associated 
to an extension of the Hasse diagram of a ranked poset $P$.  
In this section we first 
give a general algorithm for determining the Cartier divisors of $Y(\mathcal F_Q)$, 
and hence the Picard group, for arbitrary strongly-connected starred quivers $Q$. 
We then 
introduce a new \emph{canonical extension} of the ranked poset $P$.
Finally we use the canonical extension to  describe concretely the Picard group
of $Y(\mathcal F_Q)$ whenever $Q$ comes from a ranked poset.

\subsubsection{Cartier divisors in $Y(\mathcal F_Q)$ for arbitrary strongly-connected starred quivers~$Q$}

By \cref{thm:reflexive} and \cref{cor:facets}, we have the following result.
\begin{corollary} \label{cor:terminalGorenstein}
Let $\mathcal F_Q$ be the face fan of $\NP(Q)$ for a starred quiver $Q$. If $Q$ is strongly-connected then the toric variety $Y(\mathcal F_Q)$ associated to the fan $\mathcal F_Q$ is terminal Gorenstein. 
Moreover, the maximal cones of $\mathcal F_Q$ are in bijection with facet arrow-labelings of $Q$,
	where the maximal cone
associated to a facet arrow-labeling $M$ is the cone spanned by 
$$\{u_a \ \vert \ M(a) = -1\}.$$
\end{corollary}

Note that the facet arrow-labelings are in bijection with the torus-fixed points of  $Y(\mathcal F_Q)$. 

\begin{lemma}\label{l:Qrelations} 
For any subset $S$ of $\Arr(Q)$ let $\Pi(S)$ be the set of paths $\pi$ 
	with support $S$ 
in the underlying graph of $Q$ such that either $\pi$ is a closed path, or a path whose endpoints are starred vertices. The linear relations between the vectors $\{u_a\mid a\in S\}$ are generated by relations of the form
\[
\sum_{a\in\pi} \varepsilon(a)u_{a}=0,
\]
where  $\pi\in\Pi(S)$ and the sign $\varepsilon(a)=\pm 1$ is defined to be $+1$  if the orientation of the arrow $a$ agrees with the orientation of the path $\pi$, and $-1$ if the orientations are opposite.
\end{lemma}

\cref{l:Qrelations} 
 is a straightforward consequence of the definition of the vectors $u_a$. Cartier divisors on $Y(\mathcal F_Q)$ can now be described as follows. 
\begin{prop}\label{prop:Cartier}
We use the notation of \cref{l:Qrelations}. 
	Let $a\in \Arr(Q)$, and let $D_a$ denote the Weil divisor of $Y(\mathcal F_Q)$ associated to the ray of 
	$\mathcal F_Q$ spanned by $u_a$. The Weil divisor $\sum c_a D_a$ is Cartier if and only if 
	\begin{equation}\label{eq:CartierQ}
\sum_{a\in\pi} \varepsilon(a)c_{a}=0
	\end{equation}
	whenever $\pi\in \Pi(S)$, where $S=F(M)$ for 
	$M$  a facet arrow-labeling of $Q$. 
\end{prop}

\begin{proof} Let $\mathbf N_\R$ denote the vector space containing $\NP(Q)$.
Recall that a Weil divisor $\sum c_a D_a$ is Cartier if and only if for each maximal cone $\sigma$ of the fan $\mathcal F_Q$,
	the function $f_\sigma$ on primitive generators $u_a$ 
	of rays $\rho_a$ in $\sigma$ given by 
	$f_\sigma(u_a)=c_a$, 
	extends to a linear map on 
	$\mathbf N_\R$.
	Recall that we have identified the maximal cones in terms of facet arrow-labelings in 
\cref{cor:terminalGorenstein}.
	 Let $M$ be a facet arrow-labeling and let 
	 $\{u_a\mid M(a)=-1\}$ be the set of primitive vectors corresponding to rays of the associated maximal cone. Since $\NP(Q)$ is full-dimensional the maximal cone $\sigma$ is also full-dimensional and we have that $\{u_a\mid M(a)=-1\}$ spans $\mathbf N_\R$. Recalling that $
S:=\{a\in\Arr(Q)\mid M(a)=-1\}$, 
let $K$ be the kernel of the linear map
	 $\R^{S}\to \mathbf N_\R$ which sends the standard basis element $e_a\in \R^{S}$ to 
	 $u_a$. 
	 Then  $\mathbf N_\R$ is identified with the quotient $\R^{S}/K$,
	and the Weil divisor $\sum c_a D_a$ is Cartier if and only if the linear map $\R^{S}\to\R$ that takes  $e_a$ to $c_a$ is well-defined on this quotient. By \cref{l:Qrelations} applied to $S$ this is the case precisely if the condition \eqref{eq:CartierQ} holds for all  $\pi\in \Pi(S)$.     
\end{proof}

\begin{remark}\label{r:principaldiv} Recall the definition of $\mathbf N_Q$ and $\mathbf N_{Q,\R}$ and their duals $\mathbf M_Q$ and $\mathbf M_{Q,\R}$, see \cref{def:NNP} and \cref{def:proper}. By \cref{prop:NNP}, we may replace $\NP(Q)$ with the integrally equivalent $\NNP(Q)$ and consider $\F_Q$ as lying in $\mathbf N_{Q,\R}$. Thus $\mathbf M_Q$ has an interpretation as the character lattice of the torus acting on $Y(\F_Q)$.  Moreover, we may identify $\Z^{\Arr(Q)}$ with the group of torus invariant Weil divisors of $Y(\F_Q)$. The map $\mathbf M_Q\to \Z^{\Arr(Q)}$ sending a $0$-sum arrow labeling $M$ to the divisor $\sum_{a\in\Arr(Q)}M(a)D_a$ naturally identifies the group of $0$-sum arrow labelings $\mathbf M_Q$ with the group of torus-invariant principal divisors (the divisors of zeros and poles associated to characters). The class group  $\operatorname{Cl}(Y(\F_Q))$ of the toric variety $Y(\F_Q)$ is therefore given by $\Z^{\Arr(Q)}/\mathbf M_Q$, see \cite[Chapter 4]{CLS}. The Picard group is the subgroup of $\operatorname{Cl}(Y(\F_Q))$ generated by the Cartier divisor classes. In this way  \cref{prop:Cartier} allows for computation of the Picard group of the toric variety $Y(\F_Q)$. 
\end{remark}
\begin{example}\label{ex:Galkin} Consider the strongly connected starred quiver $Q$ shown in Figure~\ref{f:Galkinex}.
\begin{figure}
\[\begin{tikzcd}
	\star & \bullet & \bullet & \bullet & \star && \star & \bullet & \bullet & \bullet & \star \\
	\\
	\star & \bullet & \bullet & \bullet & \star && \star & \bullet & \bullet & \bullet & \star
	\arrow["{a_0}"{description}, from=1-1, to=1-2]
	\arrow["{a_1}"{description}, curve={height=-12pt}, from=1-2, to=1-3]
	\arrow["{a_2}"{description}, curve={height=-12pt}, from=1-3, to=1-2]
	\arrow["{a_3}"{description}, curve={height=-12pt}, from=1-3, to=1-4]
	\arrow["{a_4}"{description}, curve={height=-12pt}, from=1-4, to=1-3]
	\arrow["{a_5}"{description}, curve={height=-12pt}, from=1-4, to=1-5]
	\arrow["{a_6}"{description}, curve={height=-12pt}, from=1-5, to=1-4]
	\arrow["{-1}", Rightarrow, from=1-7, to=1-8]
	\arrow["1", curve={height=-12pt}, from=1-8, to=1-9]
	\arrow["{-1}", curve={height=-12pt}, Rightarrow, from=1-9, to=1-8]
	\arrow["{-1}", curve={height=-12pt}, Rightarrow, from=1-9, to=1-10]
	\arrow["1", curve={height=-12pt}, from=1-10, to=1-9]
	\arrow["1", curve={height=-12pt}, from=1-10, to=1-11]
	\arrow["{-1}", curve={height=-12pt}, Rightarrow, from=1-11, to=1-10]
	\arrow["{-1}", Rightarrow, from=3-1, to=3-2]
	\arrow["{-1}", curve={height=-12pt}, Rightarrow, from=3-2, to=3-3]
	\arrow["1", curve={height=-12pt}, from=3-3, to=3-2]
	\arrow["1", curve={height=-12pt}, from=3-3, to=3-4]
	\arrow["{-1}", curve={height=-12pt}, Rightarrow, from=3-4, to=3-3]
	\arrow["1", curve={height=-12pt}, from=3-4, to=3-5]
	\arrow["{-1}", curve={height=-12pt}, Rightarrow, from=3-5, to=3-4]
	\arrow["{-1}", Rightarrow, from=3-7, to=3-8]
	\arrow["1", curve={height=-12pt}, from=3-8, to=3-9]
	\arrow["{-1}", curve={height=-12pt}, Rightarrow, from=3-9, to=3-8]
	\arrow["1", curve={height=-12pt}, from=3-9, to=3-10]
	\arrow["{-1}", curve={height=-12pt}, Rightarrow, from=3-10, to=3-9]
	\arrow["{-1}", curve={height=-12pt}, Rightarrow, from=3-10, to=3-11]
	\arrow["1", curve={height=-12pt}, from=3-11, to=3-10]
\end{tikzcd}\]
\caption{A starred quiver $Q$ and three of its facet arrow-labelings\label{f:Galkinex}}
\end{figure}
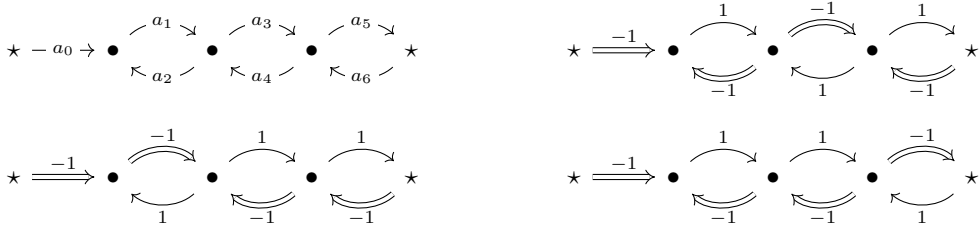
The associated toric variety $Y(\F_Q)$ is a terminal Fano $3$-fold with three singular points, namely corresponding to the maximal cones given by the facet labelings shown in \cref{f:Galkinex}. 
Recall that an element $(c_a)\in\Z^{\Arr(Q)}$ represents a Weil divisor $\sum c_aD_a$ of $Y(\mathcal F_Q)$, and the lattice of Weil divisors linearly equivalent to $0$ is given by $\mathbf M_Q\subset \Z^{\Arr(Q)}$, as explained in \cref{r:principaldiv}. 
By \cref{prop:Cartier}, the toric Cartier divisors are the divisors $\sum_{i=0}^6 c_i D_{a_i}$ whose coefficients satisfy
\[
c_0+c_1=c_4+c_6,\quad c_0+c_3=c_2+c_6,\quad c_0+c_5=c_2+c_4.
\] 
	Modulo linear equivalence
	we obtain the Picard group of $Y(\F_Q)$, which turns to be rank $1$ with generator 
\[
[D_{a_1}+D_{a_3}+D_{a_5}-D_{a_0}]=[D_{a_2}+D_{a_4}+D_{a_6}+2D_{a_0}].
\]
We see directly that the toric boundary divisor $\sum_{i=0}^6 D_{a_i}$ is Cartier and moreover $-K_{Y(\F_Q)}=[\sum_{i=0}^6 D_{a_i}]=2  [D_{a_1}+D_{a_3}+D_{a_5}-D_{a_0}]$, showing that $Y(\F_Q)$ has Fano index $2$. Additionally, $Y(\F_Q)$ has a small resolution $Y(\widehat\F_Q)$ by \cref{p:desing}, and a  Laurent polynomial superpotential which is read off the quiver $Q$ with coordinates $x_1,x_2,x_3$ corresponding to the $\bullet$-vertices $\vv_1,\vv_2,\vv_3$  in order:
\[
S_Q(x_1,x_2,x_3)=x_1+\frac{x_1}{x_2}+\frac{x_2}{x_3}+x_3+\frac{x_2}{x_1}+\frac{x_3}{x_2}+\frac{1}{x_3}.
\]
Let us change coordinates via $x_1=x y z , x_2= y z, x_3=z$. 
	Then $S_Q$ agrees precisely with the Laurent polynomial 
	$$xyz+x+y+z+\frac{1}{x}+\frac{1}{y} + \frac{1}{z}$$
	discussed in \cite[Example 5.7]{Galkin}. The constant term series of $S_Q$ (the series in $t$ made up of constant terms of $e^{t S_Q(x)}$) has a geometric interpretation in terms of different types of I-series associated to smooth anticanonical sections in  $Y(\widehat \F_Q)$ and in a smoothing of $Y(\F_Q)$, see \cite{Galkin}. 
\end{example}

\subsubsection{A canonical extension of a ranked poset}

Suppose $P$ is a ranked poset. Recall the bounded extension  $\hat P=P\cup\{\hat 0,\hat 1\}$ of $P$. If $P$ is only ranked but not graded then $\hat P$ is no longer ranked. 
We may instead consider the naive ranked extension, 
\begin{equation}\label{eq:naive}
P_{\max}=P\cup\left\{\hat 0\right\}\cup\left\{\hat 1_m\mid m\in P \text{ maximal}\right\},
\end{equation}
that adds one element $\hat 1_m$ above every maximal element $m$, 
together
with the cover relation $\hat 1_m \gtrdot m$.
From the perspective of toric geometry, however, there turns out to be a natural extension of $P$ 
that is intermediate between the bounded extension and the
 extension $P_{\max}$. 

\begin{defn}\label{d:canonical-extension}   
We define the \emph{canonical extension} $\bar P$ of a ranked poset 
$P$ as follows. Consider the starred quiver $Q_{P_{\max}}$ associated to the ranked poset $P_{\max}$. We call two starred vertices equivalent if there exists a facet arrow-labeling of $Q_{P_{\max}}$ for which they are in the same facet component, see \cref{f:canonical}. We define $\bar P$ to be the quotient poset of $P_{\max}$ obtained by identifying $\hat 1_m$ with $\hat 1_{m'}$ whenever the associated starred vertices are equivalent. \end{defn}

Let $P$ be a ranked poset, and $\bar P$ its extension from \cref{d:canonical-extension}.   Let $Q=Q_{\bar P}=(\V,\Arr(Q))$ be the starred quiver associated to $\bar P$. So $Q_{\bar P}$ has normal vertices $\V_{\bullet}=P$ and starred vertices $\V_\star=\{\star_{0},\star_1,\dotsc, \star_r\}$ associated to the elements of $\bar P\setminus P$. We assume $\star_0$ is the source vertex so that $\star_1,\dotsc, \star_r$ are the sink vertices. 

\begin{lemma}\label{l:canonicalP}
The poset $\bar P$ is ranked. Moreover, every facet-arrow labeling of $Q_{\bar P}$ has precisely $r+1$ facet components, one containing each starred vertex.  
\end{lemma}
\begin{proof}
Recall first that the poset $P_{\max}$ is a ranked poset. Fix a facet arrow-labeling $M$ of $Q_{P_{\max }}$. The facet component of $\star_0$ contains no other starred vertex, as  follows by the proof of \cref{lem:orderideal}. Now consider a facet component which contains two starred vertices $\star_i$ and $\star_j$. Then consider a non-oriented path in the facet component that connects the two starred vertices. We can work out the associated vertex labeling $L_M$ on every vertex of the facet component starting with $L_M(\star_i)=0$ using that $L_M$ increases by $1$ as we reverse along an arrow in the facet component, and decreases by $1$ as we follow an arrow upwards. In terms of the rank function, this implies that $L_M(\vv)=\rank(\hat 1_i)-\rank(\vv)$ for any $\vv\in\bar P$ lying in the facet component of $\star_i$. Since $\star_j$ is assumed to lie in the same facet component as $\star_i$, and $L_M(\star_j)=0$ as it is a starred vertex, we deduce that $\rank(\hat 1_j)=\rank(\hat 1_i)$. 

We now have that $\bar P$ is ranked, since the equivalence relation identifies only starred vertices that have the same rank.

Next we fix a facet arrow-labeling $\bar M$ of $Q_{\bar P}$. We have a natural bijection between the arrows of $Q_{\bar P}$ and those of $Q_{P_{\max}}$, and the labeling $\bar M$ transfers to a facet arrow-labeling $M$ for $Q_{P_{\max}}$. If two starred vertices of $Q_{\bar P}$ are in the same facet component for the facet arrow labeling $\bar M$ then they have representatives in $P_{\max}$ that are in the same component for $M$. Therefore actually the two starred vertices are identified in $\bar P$ by construction. It follows that the map from starred vertices of $Q_{\bar P}$ to facet components for $\bar M$ is injective. But this map is also surjective by \cref{lem:minus1}.
\end{proof}

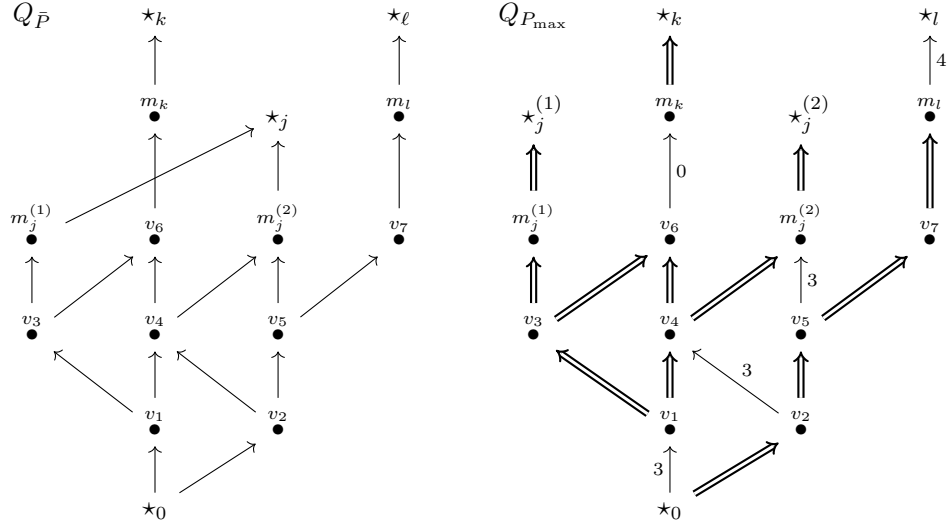
\begin{figure}[h]
\[\begin{tikzcd}
	{Q_{\bar P}} & {\star_k} && {\star_\ell} & {Q_{P_{\max }}} & {\star_k} && {\star_l} \\
	& {\overset{m_k}\bullet} & {\star_j} & {\overset{m_l}\bullet} & {\ \ \star_j^{(1)}} & {\overset{m_k}\bullet} & {\ \ \star_j^{(2)}} & {\overset{m_l}\bullet} \\
	{\overset{m_j^{(1)}}\bullet} & {\overset{v_6}\bullet} & {\overset{m_j^{(2)}}\bullet} & {\overset{v_7}\bullet} & {\overset{m_j^{(1)}}\bullet} & {\overset{v_{6}}\bullet} & {\overset{m_j^{(2)}}\bullet} & {\overset{v_7}\bullet} \\
	{\overset{v_3}\bullet} & {\overset{v_4}\bullet} & {\overset{v_5}\bullet} && {\overset{v_{3}}\bullet} & {\overset{v_{4}}\bullet} & {\overset{v_{5}}\bullet} \\
	& {\overset{v_1}\bullet} & {\overset{v_2}\bullet} &&& {\overset{v_{1}}\bullet} & {\overset{v_{2}}\bullet} \\
	& {\star_0} &&&& {\star_0}
	\arrow[from=2-2, to=1-2]
	\arrow[from=2-4, to=1-4]
	\arrow[Rightarrow, thick, from=2-6, to=1-6]
	\arrow["4"', from=2-8, to=1-8]
	\arrow[from=3-1, to=2-3]
	\arrow[from=3-2, to=2-2]
	\arrow[from=3-3, to=2-3]
	\arrow[from=3-4, to=2-4]
	\arrow[Rightarrow, thick, from=3-5, to=2-5]
	\arrow["0"', from=3-6, to=2-6]
	\arrow[Rightarrow, thick, from=3-7, to=2-7]
	\arrow[Rightarrow, thick, from=3-8, to=2-8]
	\arrow[from=4-1, to=3-1]
	\arrow[from=4-1, to=3-2]
	\arrow[from=4-2, to=3-2]
	\arrow[from=4-2, to=3-3]
	\arrow[from=4-3, to=3-3]
	\arrow[from=4-3, to=3-4]
	\arrow[Rightarrow, thick, from=4-5, to=3-5]
	\arrow[Rightarrow, thick, from=4-5, to=3-6]
	\arrow[Rightarrow, thick, from=4-6, to=3-6]
	\arrow[Rightarrow, thick, from=4-6, to=3-7]
	\arrow["3"', from=4-7, to=3-7]
	\arrow[Rightarrow, thick, from=4-7, to=3-8]
	\arrow[from=5-2, to=4-1]
	\arrow[from=5-2, to=4-2]
	\arrow[from=5-3, to=4-2]
	\arrow[from=5-3, to=4-3]
	\arrow[Rightarrow, thick, from=5-6, to=4-5]
	\arrow[Rightarrow, thick, from=5-6, to=4-6]
	\arrow["3"', from=5-7, to=4-6]
	\arrow[Rightarrow, thick, from=5-7, to=4-7]
	\arrow[from=6-2, to=5-2]
	\arrow[from=6-2, to=5-3]
	\arrow["3", from=6-6, to=5-6]
	\arrow[Rightarrow, thick, from=6-6, to=5-7]
\end{tikzcd}\]
	\caption{To the left, we have a quiver $Q_{\bar{P}}$ illustrating the canonical 
	extension of a ranked poset $P$.  To the right, the quiver corresponding to 
	$P_{\max}$ together with a facet arrow-labeling  in bold, 
	which connects the 
two vertices that are equivalent.  
	The notations are as in the proof of \cref{t:canonicalPCartier}. 
	Note that $Q=Q_{\bar P}$ and $Q_{P_{\max}}$ have the same root polytopes and associated face fan $\F_Q$\label{f:canonical}}
\end{figure}

\begin{remark}\label{r:HibiVsY}
Note that while the quiver $Q=Q_{\bar P}$ depends on the chosen extension $\bar P$ of the poset $P$, the root polytope does not, 
\[\NP(Q_{\bar P})=\NP(Q_{P_{\max}})=\NP(Q_{\hat P}),
\] 
so that the toric variety $Y(\F_Q)$ associated to its face fan really just depends on the poset $P$. The relation between our variety $Y(\F_Q)$ and the other toric variety naturally associated to the poset $P$, namely the Hibi toric variety $Y_{\OO(P)}$, is that $Y(\F_Q)$ is a small partial desingularisation of $Y_{\OO(P)}$ whenever $P$ is ranked, and if $P$ is graded the two varieties are isomorphic, see \cref{s:hibiintro}. 
\end{remark}

We will now describe the group of torus-invariant Cartier divisors and the Picard group for  $Y(\F_Q)$. 
Recall that an element $(c_a)\in\Z^{\Arr(Q)}$ represents a Weil divisor $\sum c_aD_a$ of $Y(\mathcal F_Q)$, and the lattice of Weil divisors linearly equivalent to $0$ is given by $\mathbf M_Q\subset \Z^{\Arr(Q)}$, the sublattice of integer $0$-sum arrow-labelings of $Q$ from \cref{def:proper}. 
We make the following definition. 

\begin{definition}\label{def:properCartier}
Suppose $Q$ is a strongly-connected starred quiver.  Let $M: \Arr(Q)\to \R$ be an arrow labeling of $Q$. We call $M$ an \emph{independent-sum arrow labeling} if for any oriented path from one star vertex to another, the sum $\sum_{a\in\pi} M(a)$ depends only on the endpoints.  We write $\mathbf C_Q$ for the lattice of $\Z$-valued independent-sum arrow labelings. We may consider $\mathbf C_Q$ as a sublattice of  $\Z^{\Arr(Q)}$. Note that $0$-sum arrow labelings are examples of independent-sum arrow labelings by  \cref{lem:bullet-0sum}, so that we have $\mathbf M_Q\subseteq\mathbf C_Q\subseteq \Z^{\Arr(Q)}$.
\end{definition}

The following lemma about independent-sum arrow labelings is a generalisation of  \cref{lem:bullet-0sum}.
 
\begin{lem}\label{lem:Vindependent-sum} Given a strongly-connected starred quiver $Q$ with vertices $\V=\V_\bullet\sqcup\V_\star$ and arrows $\Arr(Q)$, the lattice of independent-sum arrow labelings $\mathbf C_Q$ is the image of the map
\begin{equation}\label{e:CQaltdef}
\begin{array}{ccc}
\Z^{\V}&\to &\Z^{\Arr(Q)}\\
(\ell_\vv)_\vv&\mapsto& (c_a)_{a\in\Arr(Q)},
\end{array}
\end{equation}
where $c_a=\ell_{h(a)}-\ell_{t(a)}$ for $a\in\Arr(Q)$. Here $h(a)$ denotes the head of the arrow $a$ and $t(a)$ the tail.
\end{lem}

\begin{proof} It is immediate that the image of \eqref{e:CQaltdef} consists of independent-sum arrow labeling. Therefore it remains to show the reverse inclusion. 

Suppose $M$ is an independent-sum arrow labeling. Note that for an oriented path $\pi$ from a starred vertex to itself, the sum $\sum_{a\in\pi}M(a)$ must equal to $0$. Therefore, if there is only a single starred vertex, then any independent-sum arrow labeling is actually a $0$-sum arrow labeling, and the lemma follows from \cref{lem:bullet-0sum}. We can now proceed by induction on the number of starred vertices. Suppose we have $d+1$ starred vertices and the lemma holds whenever there are $d$ starred vertices. Amongst all of the paths in the underlying graph of $Q$ that start at one starred vertex and end at a different starred vertex, and traverse only normal vertices, pick a path $\pi$ that has a \textit{minimal} number of direction changes. We show indirectly that $\pi$ must be an oriented path in $Q$. Namely, consider the subquiver $Q'$ just consisting of the arrows of $\pi$. It has two starred vertices $\star_1\ne \star_2$. If there is a direction change along $\pi$ at a vertex $\vv$, then this vertex is a sink or a source for $Q'$. Say it is a sink. Since $Q$ is strongly connected then there is an oriented path in $Q$ from $\vv$ to a starred vertex $\star_3$. If we replace the part of $\pi$ running from $\vv$ to $\star_2$ by the new path to $\star_3$ then we have reduced the number of direction changes. If $\star_3$ happens to equal to $\star_1$, then we  replace instead the part of $\pi$ that runs from $\star_1$ to $\vv$  by the new path from $\star_1=\star_3$ to $\vv$, again reducing the number of direction changes. We argue analogously if $\vv$ is a source. Therefore we have proved that the path between distinct starred vertices with minimal number of direction changes is actually an \textit{oriented} path; it has no direction changes. 

By the argument above, there exist two starred vertices $\star_1\ne\star_2$ that are connected by an oriented path. Let $\pi$ be such an oriented path from $\star_1$ to $\star_2$. Since $M$ is an independent-sum arrow labeling, the sum $m(\star_1,\star_2):=\sum_{a\in\pi} M(a)$ is well-defined (independently of the choice of $\pi$).

We now construct a new quiver $\bar Q$ as follows. We add to $Q$ an additional arrow, namely from $\star_2$ to $\star_1$, and then we turn $\star_2$ into a normal vertex. This new quiver is still strongly connected, although it has one fewer starred vertices. We extend the arrow-labeling $M$ to a labeling $\bar M$ for $\bar Q$ by labeling the new arrow with $-m(\star_1,\star_2)$. This is now an independent-sum arrow labeling for $\bar Q$, as is straightforward to check. We can apply the induction hypothesis to obtain a vertex labeling for $\bar Q$ that maps to the arrow labeling $\bar M$. This  vertex labeling recovers all of the arrow labels of our original quiver $Q$ and shows that $M$ is in the image of the map \eqref{e:CQaltdef}. 
\end{proof}
\begin{theorem}\label{t:canonicalPCartier} 
	Let $P$ be a ranked poset, $\bar P$ its canonical extension and  let $Q=Q_{\bar P}$ be the associated starred quiver, where $Q=(\V,\Arr(Q))$ with $\V=\V_{\bullet}\cup\V_\star$.   Then the lattice $\mathbf C_Q$ of independent-sum arrow labelings agrees with the group of torus-invariant Cartier divisors of the toric variety $Y(\mathcal F_Q)$ associated to $P$. Moreover, we have a well-defined map $\mathbf C_Q\rightarrow  \Z^{\V_{\star}\setminus\{\star_0\}}$ defined by 
\begin{equation}\label{e:PicMap}\begin{tikzcd}
	M & {(\sum_{a\in\pi_i} M(a))_{\star_i\in \V_{\star}\setminus\{\star_0\}}}
	\arrow[maps to, from=1-1, to=1-2]
\end{tikzcd}\end{equation}
where $\pi_i$ is an/any oriented path from $\star_0$ to $\star_i$. The sequence
\begin{equation}\label{e:Picexactseq}
0\longrightarrow\mathbf M_Q\longrightarrow \mathbf C_Q\longrightarrow \Z^{\V_\star\setminus\{\star_0\}}\longrightarrow 0
\end{equation}
is exact and identifies the Picard group $\Pic(Y(\mathcal F_Q))$ with $\Z^{\V_\star\setminus\{\star_0\}}$. 
\end{theorem}

\begin{rem}\label{r:CartierPicP} We can also give a direct `vertex' description of the group of Cartier divisors.  Namely the map in \eqref{e:CQaltdef} induces an isomorphism
\[
\Z^{\mathcal V\setminus\{\star_0\}}\overset\sim\longrightarrow \mathbf C_Q
\]
sending $(\ell_\vv)$ to the arrow-labeling $M\in\mathbf C_Q$ with  $M(a):=\ell_{h(a)}$ if $t(a)=\star_0$, and  $M(a)=\ell_{h(a)}-\ell_{t(a)}$ otherwise. This corresponds to normalising $\ell_{\star_0}$ to $0$ in \cref{lem:Vindependent-sum} to get an isomorphism. The subgroup $\mathbf M_Q$ is then identified with $
 \Z^{\V_\bullet}\times \prod_{\V_\star\setminus\{\star_0\}}\{0\}$, the sublattice of $\Z^{\mathcal V\setminus\{\star_0\}}$ where all starred vertex coordinates $\ell_{\star_j}=0$. The exact sequence~\eqref{e:Picexactseq} becomes simply
\[
0\longrightarrow \Z^{\V_\bullet}\times \prod_{\V_\star\setminus\{\star_0\}}\{0\}\longrightarrow \Z^{\V\setminus\{\star_0\}}\longrightarrow \Z^{\V_\star\setminus\{\star_0\}}\longrightarrow 0.
\] 
Therefore from the quiver $Q_{\bar P}$ we can read off:
\begin{itemize}
	\item generators for the group of torus-invariant Weil divisors (arrows) of $Y(\F_Q)$,
\item generators for the group of torus-invariant Cartier divisors  (vertices not equal to $\star_0$),
\item generators for the Picard group (starred vertices of $Q_{\bar P}$ not equal to $\star_0$).
\end{itemize} 
The generator of the Picard group associated to a starred vertex $\star_j$ is represented by the Cartier divisor $D_j$ given by $D_j=\sum_{a\to\star_j}D_a$,, with the sum being over all arrows $a$ with target $\star_j$. The Cartier divisor associated to a normal vertex $\vv$ is in terms of Weil divisors given by the sum of the incoming arrows minus the sum of the outgoing arrows for $\vv$.  Directly in terms of the poset $\bar P$ we also have: 
\begin{itemize}
\item generators for the group of Weil divisors (covering relations in $\bar P$),
\item generators for the group of Cartier divisors  (elements of $\bar P\setminus\{\,\hat 0\,\}$),
\item generators for the Picard group (maximal elements of $\bar P$).
\end{itemize} 
\end{rem}\label{r:extensionchoice}
\begin{rem} Note that we can define $\mathbf C_Q$ for any quiver $Q$ (possibly as the image of the map \eqref{e:CQaltdef} if $Q$ is not strongly connected), and the sequence \eqref{e:Picexactseq} still makes sense. However, $\mathbf C_Q$ would not recover the group of Cartier divisors of $Y(\mathcal F_Q)$ in general. Even for $Q$ arising from a ranked poset $P$, two different extensions of $P$ give different $\mathbf C_Q$ and different sets of starred vertices, while the variety $Y(\mathcal F_Q)$ only depended on the original poset $P$. Thus, the specific construction of the canonical extension $\bar P$ of a ranked poset $P$ is important for this result. 
\end{rem}

\begin{rem}\label{r:PicvsCl} 
If we consider $Q_{\max}:=Q_{P_{\max}}$ for the finite ranked poset $P$ the analogue of the construction from \cref{t:canonicalPCartier} gives us the class group for the same toric variety $Y(\F_Q)$. Namely, the group of torus-invariant Weil divisors is identified with $\Z^{\Arr(Q_{\max})}$ and we have the exact sequence
\[
0\longrightarrow\mathbf M_{Q_{\max}}\longrightarrow \Z^{\Arr(Q_{\max})} \longrightarrow \operatorname{Cl}(Y(\F_Q))\longrightarrow 0.
\]
The class group as the cokernel of the map $\mathbf M_{Q_{\max}}\longrightarrow \Z^{\Arr(Q_{\max})}$ 
	is seen to be isomorphic to the group $\Z^d$, where $d$ is the number of 
	maximal elements in $P_{\max}$, equivalently in $P$.
	Thus all in all we have that generators of the class group of $Y(\F_Q)$ are in bijection with maximal elements of $P_{\max}$, and generators of the Picard group of $Y(\F_Q)$ with maximal elements of the canonical extension $\bar P$. 
\end{rem}

\begin{proof}[Proof of \cref{t:canonicalPCartier}] 
Suppose $(c_a)\in\mathbf C_Q$. Using \cref{lem:Vindependent-sum}, choose an $(\ell_{\vv})_\vv\in\Z^\V$ mapping to $(c_a)$ under \eqref{e:CQaltdef}.    From this presentation of $(c_a)_{a\in\Arr(Q)}$ it follows that $\sum c_aD_a$ is a Cartier divisor in $Y(\F_Q)$, as a direct consequence of the characterisation of Cartier divisors in \cref{prop:Cartier}. We now prove that all Cartier divisors lie in $\mathbf C_Q$.

Suppose $\sum_{a\in\Arr} c_aD_a$ is a fixed Cartier divisor for $Y(\mathcal F_Q)$. Pick an unoriented path $\pi$ in $Q_{\bar P}$ as in \cref{l:Qrelations} with $S=\Arr(Q)$. Suppose first that the path only involves vertices from $\V_\bullet\cup\{\star_0\}$. We construct a facet arrow-labeling for $Q_{\bar P}$ as follows. Define a vertex labeling $L$ by $L(\vv)=-\rank(\vv)$ if $\vv\in\V_\bullet$, and $L(\vv)=0$ if $\vv\in\V_\star$. The associated arrow-labeling $M$ labels all arrows not pointing to a $\star$-vertex by $-1$, while an arrow pointing to a vertex $\star_j$ has label $M(a)=\rank(\star_j)-1$. This is clearly a facet arrow-labeling. The facet component of $\star_{0}$ contains all of the arrows that point to normal vertices and, in particular, entirely contains the path $\pi$. It follows from the characterisation of Cartier divisors in \cref{prop:Cartier} that for any such path $\pi$ the relation 
\begin{equation}\label{e:bulletrel}
\sum_{a\in\pi}\epsilon(a) c_{a}=0
\end{equation}
holds. Given $\vv\in\V_\bullet$, pick an oriented path $\pi_\vv$ from $\star_0$ to $\vv$ and define $\ell_\vv:=\sum_{a\in\pi_\vv}c_a$. Also set $\ell_{\star_0}=0$. It follows from \eqref{e:bulletrel} that the element $(\ell_\vv)\in\Z^{\V_{\bullet}\cup\{\star_0\}}$ is well-defined independently of the paths chosen. Moreover, it determines $c_a$ for any arrow $a$ not pointing to a starred vertex by $c_a=\ell_{h(a)}-\ell_{t(a)}$. 

Consider a sink vertex $\star_j$ and pick an arrow pointing to it, 
\[
\bullet_{m^{(1)}_j}\overset {a^{(1)}}\longrightarrow \star_j.
\]
Set $\ell_{\star_j}=\ell_{m^{(1)}_j}+c_{a^{(1)}}$. We now check that  if there is another arrow $a^{(2)}$ ending in $\star_j$ the analogous quantity $\ell_{m^{(2)}_j}+c_{a^{(2)}}$ agrees, so that $\ell_{\star_j}$ is well-defined.

In  the related quiver $Q_{P_{\max}}$  these arrows point to different starred vertices, that we shall call $\star_j^{(1)}$ and $\star_j^{(2)}$. By definition of $\bar P$ there exists a facet arrow-labeling $M$ of $Q_{P_{\max}}$ and an unoriented path $\pi_j=(a_1,\dotsc, a_k)$ from $\star_j^{(1)}$ to $\star_j^{(2)}$ for which all arrows are labeled by $-1$. We may suppose $a_1=a^{(1)}$ and $a_k=a^{(2)}$, and we have $\varepsilon(a_1)=-1$ and $\varepsilon(a_k)=1$. Now \cref{prop:Cartier} applies and we can rewrite \eqref{eq:CartierQ}, separating out the first and last summand, to get
\[
c_{a^{(1)}}=c_{a^{(2)}}+\sum_{i=2}^{k-1}\varepsilon(a_i)c_{a_i}. 
\]
The path $(a_2,\dotsc,a_{k-1})$ runs from $m_j^{(1)}$ to $m_j^{(2)}$ and can be assumed to only traverse normal vertices. Therefore we can replace $\sum_{i=2}^{k-1}\varepsilon(a_i)c_{a_i}$ by $\ell_{m_j^{(2)}}-\ell_{m_j^{(1)}}$. It follows that
\[
c_{a^{(1)}}+\ell_{m_j^{(1)}}=c_{a^{(2)}}+\ell_{m_j^{(2)}}.
\]
Therefore we have defined an element $(\ell_\vv)_\vv\in\Z^{\V}$ that maps to $(c_a)_a$. This implies the independent-sum condition and proves that $(c_a)_a$ lies in $\mathbf C_Q$. Thus $\mathbf C_Q$ agrees with the group of Cartier divisors. 

It follows immediately from the definition of $\mathbf C_Q$ that the map \eqref{e:PicMap} is well-defined. The divisors $D_j$ from \cref{r:CartierPicP} map to the standard generators of $\Z^{\V_\star\setminus\{0\}}$, which implies surjectivity. It follows from the definitions that the kernel of \eqref{e:PicMap} is precisely $\mathbf M_Q$. Thus \eqref{e:Picexactseq} is an exact sequence and the rest of the theorem follows.     
\end{proof}

\begin{cor} Let $P$ be a ranked poset with order polytope $\mathcal O(P)$. The Hibi projective toric variety $Y_{\mathcal O(P)}$ associated to the order polytope $\mathcal O(P)$ has a small toric partial desingularization  by a terminal Gorenstein Fano variety whose Picard rank equals the number of maximal elements in the canonical extension $\bar P$ of $P$. 
\end{cor}
\begin{proof}
The 
partial 
desingularization 
is given by $Y(\mathcal F_Q)$ for the quiver $Q=Q_{\hat P}$. 
Namely the face fan $\mathcal F_Q$ of $\NP(Q)$ refines the normal fan of $\mathcal O(P)$ by \cref{thm:refine}. Since no rays are added, this desingularization is small.
And $Y(\mathcal F_Q)$ is terminal Gorenstein by \cref{cor:terminalGorenstein}. The canonical extension $\bar P$ defines a different quiver $Q_{\bar P}$ but with the same root polytope;
	that is, $\NP(Q_{\hat P})=\NP(Q_{\bar P})$. Therefore the Picard rank of $Y(\mathcal F_Q)$ is given in \cref{t:canonicalPCartier}, and we see that it agrees with the number of maximal elements of $\bar P$.  
\end{proof}
Let $P_D$ denote the polytope associated to an ample divisor $D$ of a projective toric variety, as in \cite{CLS}.
\begin{cor} \label{cor:newW}
	Let $P$ be a finite ranked poset with canonical extension $\bar P$ and  maximal elements of $\bar P$ denoted $\{\star_1,\dotsc,\star_s\}$. For each maximal element $\star_i$ of $\bar P$ we have a generator $[D_i]$ of the Picard group of $Y(\F_Q)$ as in \cref{r:CartierPicP}. 
Consider the quiver Laurent polynomial $S_{Q_{\bar P}}$ with one quantum parameter $q_{i}$ associated to every sink $\star$-vertex $\star_i$. 
Suppose $D=\sum_i r_iD_i$ is ample, then associated polytope $P_D$ 
	is equal to 
	the superpotential polytope $\Gamma_{Q_{\bar P}}^{\mathbf{r}}$, where $\bold r=(r_1,\dots, r_s)$.
\end{cor}

\begin{remark}
\cite[Proposition 19]{AltmannvanStraten} gives a description of the Cartier divisors and Picard group of the projective toric variety associated to a flow polytope; 
it also gives a concrete description of the Picard group
in the case of \emph{flag quivers} 
\cite[Definition 21]{AltmannvanStraten}.
We note that in the special case when our quiver $Q$ comes from a \textit{graded, planar} poset, its dual quiver $Q^{\vee}$ is a flag quiver.
\end{remark}

We now make use of the resolution of singularities of $Y(\F_Q)$ from \cref{p:desing}. 
\begin{prop}
	\label{p:Hibidesing} 
	Let $P$ be a finite, ranked poset with order polytope $\OO(P)$. The Hibi toric variety $Y_{\OO(P)}$ with fan $\NF(\OO(P))$ has a small resolution of singularities $Y(\widehat\F_Q)\to Y(\F_Q)$. The Picard rank $\rho$ of $Y(\widehat\F_Q)$ is the number of maximal elements in $P$.
\end{prop}

\begin{proof}
Let $Q=Q_{\bar P}$ and consider the fan $\widehat\F_Q$ from \cref{p:desing}. The small desingularisation of $Y_{\OO(P)}$ is the composition of the small desingularisation $Y(\widehat\F_Q)\to Y(\F_Q)$ in \cref{p:desing} and the small partial desingularisation $Y(\F_Q)\to Y(\NF(\OO(P)))$ from \cref{t:canonicalPCartier}. Now $Y(\widehat\F_Q)$ has an isomorphic group of torus-invariant Weil divisors to $Y(\F_Q)$, but since it is smooth this identifies it with the group of torus-invariant Cartier divisors of $Y(\widehat\F_Q)$. Therefore the calculation of the class group of $Y(\F_Q)$ from \cref{r:PicvsCl} computes the Picard group of $Y(\widehat \F_Q)$. 
\end{proof}

\begin{remark} In the \textit{graded} case the construction of a small resolution for a Hibi toric variety is discussed in \cite[Section~2.4]{Miura:CYinHibi}.
\end{remark}

\vskip .2cm

\noindent{\bf Acknowledgements:~}
The first author thanks Tim Magee for helpful discussions about toric geometry. The second author thanks Karola Meszaros and Alejandro Morales for helpful conversations about root polytopes and flow polytopes, and Elana Kalashnikov for helpful discussions about mirror symmetry for toric varieties. 
\vskip .2cm


\vskip .2cm
\noindent {\bf Financial Support:~} The first author is supported by EPSRC grant EP/V002546/1.
The second author was supported by the National Science Foundation under Award No.
 DMS-2152991. 
 Any opinions, findings, and conclusions or recommendations expressed in this material are
those of the author(s) and do not necessarily reflect the views of the National Science
Foundation.

\vskip .2cm

\bibliographystyle{plain}
\bibliography{bibliography}

\end{document}